\documentclass[onecolumn]{autart}

\usepackage{amssymb, amsmath, mathrsfs, enumitem, color, wasysym, url, psfrag}
\usepackage[draft]{graphicx}

\definecolor{mred}{rgb}{0.6, 0, 0}
\definecolor{mgreen}{rgb}{0, 0.5, 0}
\definecolor{mblue}{rgb}{0, 0, 0.5}
\definecolor{mcyan}{rgb}{0, 0.5, 0.5}

\def\NN{\mathbb{N}}

\def\RR{\mathbb{R}}
\def\EE{\mathbb{E}}
\def\PP{\mathbb{P}}
\def\F{\mathscr{F}}
\def\U{\mathscr{U}}
\def\D{\mathscr{D}}

\def\x0{x_0}

\newcommand{\mb}{\mathbb}
\newcommand{\mc}{\mathcal}

\newcommand{\mrm}{\mathrm}
\newcommand{\inprod}[2]{\left\langle{#1},{#2}\right\rangle}
\newcommand{\norm}[1]{\left\lVert{#1}\right\rVert}
\newcommand{\epower}[1]{\mathrm e^{#1}}
\newcommand{\eps}{\varepsilon}
\newcommand{\indic}[1]{\mathbf 1_{#1}}

\DeclareMathOperator{\erf}{erf}
\DeclareMathOperator{\erfc}{erfc}

\newtheorem{proposition}{Proposition}
\newtheorem{lemma}[proposition]{Lemma}
\newtheorem{assumption}{Assumption}
\newtheorem{corollary}[proposition]{Corollary}
\theoremstyle{remark}
\newtheorem{example}[proposition]{Example}
\newtheorem{remark}[proposition]{Remark}
\newenvironment{proof}{\par \noindent \textit{Proof:}}{\hfill $\square$ \\ \par \noindent}

\begin{document}

\begin{frontmatter}

\title{On convex problems in chance-constrained stochastic model predictive control\thanksref{footnoteinfo}}
\thanks[footnoteinfo]{This paper was not presented at any IFAC
meeting. Research supported in part by the Swiss National Science
Foundation under grant 200021- 122072. Corresponding author
Eugenio Cinquemani. Tel. +41 (0)44 632 86 61; Fax +41 (0)44 632 12
11.}
\author[ETH]{Eugenio Cinquemani}\ead{cinquemani@control.ee.ethz.ch},
\author[STANFORD]{Mayank Agarwal}\ead{mayankag@stanford.edu},
\author[ETH]{Debasish Chatterjee}\ead{chatterjee@control.ee.ethz.ch},
\author[ETH]{John Lygeros}\ead{lygeros@control.ee.ethz.ch}

\address[ETH]{Automatic Control Laboratory, Physikstrasse 3, ETH Z\"urich, 8092 Z\"urich, Switzerland}
\address[STANFORD]{Department of Electrical Engineering, Stanford University, California, USA}


\begin{abstract}
We investigate constrained optimal control problems for linear
stochastic dynamical systems evolving in discrete time. We
consider minimization of an expected value cost over a finite
horizon. Hard constraints are introduced first, and then
reformulated in terms of probabilistic constraints. It is shown
that, for a suitable parametrization of the control policy, a wide
class of the resulting optimization problems are convex, or admit
reasonable convex approximations.
\end{abstract}

\end{frontmatter}

\section{Introduction}
This work stems from the attempt to address the optimal
infinite-horizon constrained control of discrete-time stochastic
processes by a model predictive control
strategy~\cite{ECC,BoundedControl,HitTime,ReachAvoid,kouvaritakissMPCIneqconstraints,vanHessemFullSolution,aldenRH,batinaPhDthesis,bertsekas}.
We focus on linear dynamical systems driven by stochastic noise
and a control input, and consider the problem of finding a control
policy that minimizes an expected cost function while
simultaneously fulfilling constraints on the control input and on
the state evolution. In general, no control policy exists that
guarantees satisfaction of deterministic (hard) constraints over
the whole infinite horizon. One way to cope with this issue is to
relax the constraints in terms of probabilistic (soft)
constraints~\cite{Stanfordgang,primbs}. This amounts to requiring
that constraints will not be violated with sufficiently large
probability or, alternatively, that an expected reward for the
fulfillment of the constraints is kept sufficiently large.

Two considerations lead to the reformulation of an infinite
horizon problem in terms of subproblems of finite horizon length.
First, given any bounded set (e.g. a safe set), the state of a
linear stochastic dynamical system is guaranteed to exit the set
at some time in the future with probability one whatever the
control policy. Therefore, soft constraints may turn the original
(infeasible) hard-constrained optimization problem into a feasible
problem only if the horizon length is finite. Second, even if the
constraints are reformulated so that an admissible
infinite-horizon policy exists, the computation of such a policy
is generally intractable. The aim of this note is to show that,
for certain parameterizations of the policy
space~\cite{Lofberg:03,ref:ben-tal,goulart} and the constraints,
the resulting finite horizon optimization problem is tractable.

An approach to infinite horizon constrained control problems that
has proved successful in many applications is model predictive
control~\cite{maciejowski}. In model predictive control, at every
time $t$, a finite-horizon approximation of the infinite-horizon
problem is solved but only the first control of the resulting
policy is implemented. At the next time $t+1$, a measurement of
the state is taken, a new finite-horizon problem is formulated,
the control policy is updated, and the process is repeated in a
receding horizon fashion. Under time-invariance assumptions, the
finite-horizon optimal control problem is the same at all times,
giving rise to a stationary optimal control policy that can be
computed offline.

Motivated by the previous considerations, here we study the
convexity of certain stochastic finite-horizon control problems
with soft constraints. Convexity is central for the fast
computation of the solution by way of numerical procedures, hence
convex formulations~\cite{bertsimas2007} or convex
approximations~\cite{nemshap,bertsim} of the stochastic control
problems are commonly sought. However, for many of the classes of
problems considered here, tight convex approximations are usually
difficult to derive. One may argue that non-convex problems can be
tackled by randomized
algorithms~\cite{spallbook,tempoetal,sagarpaper}. However,
randomized solutions are typically time-consuming and can only
provide probabilistic guarantees. In particular, this is critical
in the case where the system dynamics or the problem constraints
are time-varying, since in that case optimization must be
performed in real-time.

Here we provide conditions for the convexity of chance constrained
stochastic optimal control problems. We derive and compare several
explicit convex approximations of chance constraints for Gaussian
noise processes and for polytopic and ellipsoidal constraint
functions. Finally, we establish conditions for the convexity of a
class of expectation-type constrains that includes standard
integrated chance constraints~\cite{haneveldICC,haneveldICC06} as
a special case. For integrated chance constrains on Gaussian
processes with polytopic constraint functions, an explicit
formulation of the optimization problem is also derived.

The optimal constrained control problem we concentrate on is
formulated in Section~\ref{sec:ps}. A convenient parametrization
of the control policies and the convexity of the objective
function are discussed at this stage.
Next, two probabilistic formulations of the constraints and
conditions for the convexity of the space of admissible control
policies are discussed: Section~\ref{sec:cc} is dedicated to
chance constraints, while Section~\ref{sec:icc} is dedicated to
integrated chance constraints. In~Section~\ref{sec:num}, numerical
simulations are reported to illustrate and discuss the results of
the paper.

\section{Problem statement}
\label{sec:ps}

Let $\NN=\{1,2,\ldots\}$ and $\NN_0\triangleq \NN \cup \{0\}$.
Consider the following dynamical model: for $t\in\NN_0$,
\begin{equation}
\label{eq:system} x(t+1)=A x(t)+ Bu(t)+ w(t)\end{equation} where
$x(t)\in\RR^n$ is the state, $u(t)\in \RR^m$ is the control input,
$A\in\RR^{n\times n}$, $B\in\RR^{n\times m}$, and $w(t)$ is a
stochastic noise input defined on an underlying probability space
$(\Omega, \mathfrak{F}, \PP)$. No assumption on the probability
distribution of the process $w$ is made at this stage. We assume
that at any time $t\in\NN_0$, $x(t)$ is observed exactly and that,
for given $x_0\in\RR^n$, $x(0)=x_0$.
\par Fix a horizon length $N\in\NN$. The evolution of the system
from $t=0$ through $t=N$ can be described in compact form as
follows:
\begin{equation}
\label{eq:compactdyn}
\bar{x}=\bar{A}\x0+\bar{B}\bar{u}+\bar{D}\bar{w},
\end{equation}
where
$$\bar{x}\triangleq\begin{bmatrix}
x(0) \\ x(1)  \\ \vdots \\ x(N)
\end{bmatrix},\qquad \bar{u}\triangleq \begin{bmatrix}
u(0) \\ u(1) \\ \vdots \\ u(N-1)
\end{bmatrix},\qquad
\bar{w}\triangleq
\begin{bmatrix}
w(0) \\ w(1) \\ \vdots \\ w(N-1)
\end{bmatrix},\qquad \bar{A}\triangleq
\begin{bmatrix}
I_n \\ A \\ \vdots \\ A^N
\end{bmatrix},$$
$$
\bar{B}\triangleq
\begin{bmatrix}
0_{n\times m} &\cdots &\cdots & 0_{n\times m} \\
B &\ddots &&\vdots\\
AB & B &\ddots & \vdots\\
\vdots && \ddots &0_{n\times m}\\
A^{N-1} B & \cdots & AB& B
\end{bmatrix}, \qquad
\bar{D}\triangleq
\begin{bmatrix}
0_{n} &  \cdots &\cdots& 0_{n} \\
I_n & \ddots & &\vdots \\
A & I_n &\ddots &\vdots \\
\vdots && \ddots & 0_n \\
A^{N-1} &\cdots & A& I_n
\end{bmatrix}.
$$
Let $V:\RR^{(N+1)n\times Nm}\to\RR$ and $\eta:\RR^{(N+1)n\times
Nm}\to \RR^{r}$, with $r\in \NN$, be measurable functions.
We are interested in constrained optimization
problems of the following kind:
\begin{equation}
\label{eq:problem}
\begin{aligned}
\inf_{\bar{u}\in \U}\quad&\EE[V(\bar{x},\bar{u})] \\
\textrm{subject to}\quad&\textrm{(\ref{eq:compactdyn})}\quad\textrm{and}\quad
\eta(\bar{x},\bar{u})\leq 0
\end{aligned}
\end{equation}
where the expectation $\EE[\cdot]$ is defined in terms of the
underlying probability space $(\Omega, \mathfrak{F}, \PP)$, $\U$
is a class of causal deterministic state-feedback control policies
and the inequality in~(\ref{eq:problem}) is interpreted
componentwise.
\begin{example}\rm
\label{thm:examples}
In the (unconstrained) linear stochastic
control problem~\cite{astrom}, $w$ is Gaussian white noise and the
aim is to minimize
$$\EE\left[\sum_{t=0}^{N-1} \left(x^T(t)Q(t)x(t)+u^T(t)R(t)u(t)\right)+x^T(N)Q(N)x(N)\right],$$
where the matrices $Q(t)\in\RR^{n\times n}$ and
$R(t)\in\RR^{m\times m}$ are positive definite for all $t$, with
respect to causal feedback policies subject to the system
dynamics~(\ref{eq:system}). This problem fits easily in our
framework; it suffices to define
\begin{equation}
\label{eq:quadratic}
V(\bar{x},\bar{u})=\begin{bmatrix}\bar{x}^T & \bar{u}^T
\end{bmatrix}
M
\begin{bmatrix}\bar{x} \\ \bar{u}
\end{bmatrix},
\end{equation}
with
$M=\textrm{diag}\big(Q(0),Q(1),\ldots,Q(N),R(0),R(1),\ldots,
R(N-1)\big)>0$ (the notation $M>0$ indicates that $M$ is a
positive definite matrix). In our framework, though, the input
noise sequence may have an arbitrary correlation structure, the
cost function may be non-quadratic and, most importantly,
constraints may be present.
\par Standard constraints on the state and the input are also
formulated easily. For instance, sequential ellipsoidal
constraints of the type
\begin{align*}
\begin{bmatrix} x^T(t) & u^T(t)\end{bmatrix}S(t)
\begin{bmatrix}x(t) \\ u(t)\end{bmatrix}&\leq 1,\qquad t=0,1,\ldots,
N-1, \\
x^T(N)S(N) x(N)&\leq 1,
\end{align*}
with $0<S(t)\in\RR^{(n+m)\times (n+m)}$ for $t=0,1,\ldots, N-1$
and $0<S(N)\in\RR^{n\times n}$, are captured by the definition
$$\eta(\bar{x},\bar{u})=\begin{bmatrix}\eta_0(\bar{x},\bar{u}) & \eta_1(\bar{x},\bar{u}) & \cdots & \eta_N(\bar{x}) \end{bmatrix}^T$$
where, for $t=0,1,\ldots, N$,
$$\eta_t(\bar{x},\bar{u})=\begin{bmatrix}\bar{x}^T &
\bar{u}^T
\end{bmatrix}
\Xi_t
\begin{bmatrix}\bar{x} \\ \bar{u}
\end{bmatrix}-1,
$$
and each matrix $\Xi_t$ is immediately constructed in terms of the
$S(t)$. Our framework additionally allows for cross-constraints
between states and inputs at different times.
\end{example}

\subsection{Feedback from the noise input} \label{sec:if}

By the hypothesis that the state is observed without error, one
may reconstruct the noise sequence from the sequence of observed
states and inputs by the formula
\begin{equation}
\label{eq:noiserec}
w(t)=x(t+1)-A x(t) -B u(t),\qquad t\in\NN_0.
\end{equation}
In light of this, and
following~\cite{Lofberg:03,goulart,ref:ben-tal}, we shall consider
policies of the form:
\begin{equation}
\label{eq:controlpolicy}
u(t)=\sum_{i=0}^{t-1} G_{t,i} w(i) + d_t,
\end{equation}
where the feedback gains $G_{t,i}\in\RR^{m\times n}$ and the
affine terms $d_t\in\RR^m$ must be chosen based on the control
objective. With this definition, the value of $u$ at time $t$
depends on the values of $w$ up to time $t-1$.
Using~(\ref{eq:noiserec}) we see that $u(t)$ is a function of the
observed states up to time $t$. It was shown in~\cite{goulart}
that there exists a (nonlinear) bijection between control policies
in the form~(\ref{eq:controlpolicy}) and the class of affine state
feedback policies. That is, provided one is interested in affine
state feedback policies, parametrization~(\ref{eq:noiserec})
constitutes no loss of generality. Of course, this choice is
generally suboptimal, since there is no reason to expect that the
optimal policy is affine, but it will ensure the tractability of a
large class of optimal control problems. In compact notation, the
control sequence up to time $N-1$ is given by
\begin{equation}
\label{e:ubardef}
\bar{u}=\bar{G}\bar{w}+\bar{d},
\end{equation}
where
\begin{equation}
\label{e:parstruc}
\bar{G}\triangleq
\begin{bmatrix}
0_{m\times n} \\
G_{1,0} &0_{m\times n} \\
\vdots &\ddots &\ddots \\
G_{N-1,0} &\cdots &  G_{N-1,N-2} & 0_{m\times n}
\end{bmatrix},\qquad
\bar{d}\triangleq
\begin{bmatrix}
d_0 \\ d_1 \\ \vdots \\ d_{N-1}
\end{bmatrix};
\end{equation}
note the lower triangular structure of $\bar G$ that enforces
causality. The resulting closed-loop system dynamics can be
written compactly as the equality constraint
\begin{equation}
    \label{eq:cloopsys}
    \bar{x}=\bar{A}\x0+\bar{B}(\bar{G}\bar{w}+\bar{d})+\bar{D}\bar{w}.
\end{equation}

Let us denote the parameters of the control policy by $\theta=
(\bar G, \bar d)$ and write $(\bar x_\theta, \bar u_\theta)$ to
emphasize the dependence of $\bar{x}$ and $\bar{u}$ on $\theta$.
From now on we will consider the optimization problem
\begin{align}
\inf_{\theta\in \Theta}\quad&\EE[V(\bar{x}_\theta,\bar{u}_\theta)] \label{eq:objective} \\
\textrm{subject to}\quad&\eqref{e:ubardef},~\eqref{eq:cloopsys}\textrm{ and} \\
&\eta(\bar{x}_\theta,\bar{u}_\theta)\leq 0, \label{eq:constraint}
\end{align}
where $\Theta$ is the linear space of optimization parameters in
the form~\eqref{e:parstruc}.

\begin{remark}\rm
With the above parametrization of the control policy, both $\bar
u_\theta$ and $\bar x_\theta$ are affine functions of the
parameters $\theta$ (for fixed $\bar w$) and of the process noise
$\bar{w}$ (for fixed $\theta$). Most of the results developed
below rely essentially on this property. It was noticed
in~\cite{Ben-TaletAl} that a parametric causal feedback control
policy with the same property can be easily defined based on
indirect observations of the state, provided the measurement model
is linear. The method enables one to extend the results of this
paper to the case of linear output feedback. For the sake of
conciseness, this extension will not be pursued here.
\end{remark}

\subsection{Optimal control problem with relaxed constraints}

In general, no control policy can ensure that the
constraint~(\ref{eq:constraint}) is satisfied for all outcomes of
the stochastic input~$\bar{w}$. In the standard LQG setting, for
instance, any nontrivial constraint on the system state would be
violated with nonzero probability. We therefore consider relaxed
formulations of the constrained optimization
problem~\eqref{eq:objective}--\eqref{eq:constraint} of the form
\begin{align}
\inf_{\theta\in \Theta}\quad&\EE[V(\bar{x}_\theta,\bar{u}_\theta)] \label{eq:objective-2} \\
\textrm{subject to}\quad&\eqref{e:ubardef},~\eqref{eq:cloopsys}\textrm{ and}  \\
&\EE [\phi\circ \eta(\bar{x}_\theta,\bar{u}_\theta)] \leq 0, \label{eq:constraint-2}
\end{align}
where $\phi:\RR^r\to \RR^R$, with $R\in\NN$, is a convenient
measurable function and the inequality is again interpreted
componentwise. For appropriate choices of $\phi$, this formulation
embraces most common probabilistic constraint relaxations,
including chance constraints (see e.g.~\cite{nemshap}), integrated
chance constraints~\cite{haneveldICC,haneveldICC06}, and
expectation constraints~(see e.g.~\cite{primbs}). \par
We are
interested in the convexity of the optimization
problem~\eqref{eq:objective-2}--\eqref{eq:constraint-2}. First we
establish a general convexity result.
\begin{proposition}\label{thm:phietaconvex}
Let $(\Omega,\mathfrak{F},\PP)$ be a probability space, $\Theta$
be a convex subset of a vector space and $\D\subseteq \RR$ be
convex. Let $\gamma:\Omega\times\Theta\to \D$ and
$\varphi:\D\to\RR$ be measurable functions and define
$$J(\theta)\triangleq\EE[\varphi\circ\gamma(\omega,\theta)].$$
Assume that:
\begin{enumerate}[label=\emph{(\roman*)}, leftmargin=*, widest=iii, align=right]
\item \label{item:i} the mapping $\gamma(\omega,\cdot):\Theta
    \to \RR$ is convex for almost all $\omega\in\Omega$;
\item $\varphi$ is monotone nondecreasing and convex; \item
    $J(\theta)$ is finite for all $\theta\in\Theta$.
\end{enumerate}
Then the mapping $J:\Theta\to\RR$ is convex.
\end{proposition}
\begin{proof}
Fix a generic $\omega\in\Omega$. Since $\gamma(\omega,\theta)$ is
convex in $\theta$ and $\varphi$ is monotone nondecreasing, for
any $\theta,\theta'\in\Theta$ and any $\alpha\in [0,1]$,
$$\varphi\big(\gamma(\omega,\alpha \theta+(1-\alpha)\theta')\big)\leq \varphi\big(\alpha \gamma(\omega,
\theta)+(1-\alpha)\gamma(\omega,\theta')\big).$$ Moreover, since
$\varphi$ is convex, $$\varphi\big(\alpha \gamma(\omega,
\theta)+(1-\alpha)\gamma(\omega,\theta')\big)\leq
\alpha\varphi\big( \gamma(\omega,
\theta)\big)+(1-\alpha)\varphi\big(\gamma(\omega,\theta')\big).$$
Since these inequalities hold for almost all $\omega\in\Omega$, it
follows that
\begin{align*}
\EE[\varphi\big(\gamma(\omega,\alpha
\theta+(1-\alpha)\theta')\big)] \leq &\,\, \EE[\alpha\varphi\big(
\gamma(\omega,
\theta)\big)+(1-\alpha)\varphi\big(\gamma(\omega,\theta')\big)] \\
=&\,\, \alpha\EE[\varphi\big( \gamma(\omega,
\theta)\big)]+(1-\alpha)\EE[\varphi\big(\gamma(\omega,\theta')\big)],
\end{align*} which proves the assertion.
\end{proof}%
Assumption~(iii) can be replaced by either of the following:
\begin{itemize}[leftmargin=0.35in, align=right]
\item[(iii$'$)] $J(\theta)\in \RR\cup\{+\infty\}$, $\forall
    \theta\in\Theta$.
\item[(iii$''$)] $J(\theta)\in \RR\cup\{-\infty\}$, $\forall
    \theta\in\Theta$.
\end{itemize}

Let us now make the following standing assumption.
\begin{assumption}
\label{a:convexobj} $V(\bar{x},\bar{u})$ is a convex function of
$(\bar{x},\bar{u})$ and $\EE[V(\bar{x}_\theta,\bar{u}_\theta)]$ is
finite for all $\theta\in\Theta$.
\end{assumption}
\begin{proposition}
\label{proposition:CostConvexity}
Under Assumption~\ref{a:convexobj},
$\EE[V(\bar{x}_\theta,\bar{u}_\theta)]$ is a convex function of
$\theta$.
\end{proposition}
\begin{proof}
First, note that the set $\Theta$ of admissible parameters
$\theta$ is a linear space. Let us write $\bar{w}(\omega)$
$\bar{x}_\theta(\omega)$ and $\bar{u}_\theta(\omega)$ to express
the dependence of $\bar{w}$, $\bar{x}_\theta$ and $\bar{u}_\theta$
on the random event $\omega\in\Omega$. Fix $\omega$ arbitrarily.
Since the
mapping 
$$
\theta\mapsto
\begin{bmatrix}
\bar{x}_\theta(\omega) \\
\bar{u}_\theta(\omega)
\end{bmatrix}=
\begin{bmatrix}
\bar{A}\x0+\bar{B}(\bar{G}\bar{w}(\omega)+\bar{d})+\bar{D}\bar{w}(\omega) \\
\bar{G}\bar{w}(\omega)+\bar{d}
\end{bmatrix}
$$
is affine and the mapping $(\bar{x},\bar{u})\mapsto
V(\bar{x},\bar{u})$ is assumed convex, their combination
$\theta\mapsto V(\bar{x}_\theta,\bar{u}_\theta)$ is a convex
function of $\theta$.
Then, the result follows from Proposition~\ref{thm:phietaconvex}
with
$\gamma(\omega,\theta)=V\big(\bar{x}_\theta(\omega),\bar{u}_\theta(\omega)\big)$
and $\varphi$ equal to the identity map.
\end{proof}%
By virtue of the alternative assumptions (iii$'$) and (iii$''$) of
Proposition~\ref{thm:phietaconvex}, the requirement that
$\EE[V(\bar{x}_\theta,\bar{u}_\theta)]$ be finite for all $\theta$
may be relaxed. A sufficient requirement is that there exist no
two values $\theta$ and $\theta'$ such that
$J_{x_0}(\theta)=+\infty$ and $J_{x_0}(\theta')=-\infty$. In
particular, the result applies to quadratic cost functions of the
type~(\ref{eq:quadratic}), with $M\geq 0$.
\par
In general, the relaxed constraint~\eqref{eq:constraint-2} is
nonconvex even if the components $\eta_i:\RR^{(N+1)n\times Nm}\to
\RR$, with $i=1,\ldots, r$, of the vector function $\eta$ are
convex. In the next sections we will study the convexity and
provide convex approximations of~\eqref{eq:constraint-2} for
different approaches to probabilistic relaxation of hard
constraints, i.e. for different choices of the function $\phi$.

\section{Chance Constraints}
\label{sec:cc}

For a given $\alpha \in\;]0,1[$, we relax the hard constraint
$\eta(\bar x_\theta, \bar u_\theta)\leq0$ by requiring that it be
satisfied with probability $1-\alpha$. Hence we address the
optimization problem
\begin{align}
\inf_{\theta\in \Theta}\quad&\EE[V(\bar{x}_\theta,\bar{u}_\theta)] \label{e:objectiveCC} \\
\textrm{subject to}\quad&\eqref{e:ubardef},~\eqref{eq:cloopsys}\textrm{ and} \label{e:dynconstr} \\
&\PP(\eta(\bar x_\theta, \bar u_\theta) \le 0)\geq 1-\alpha \label{e:constr}.
\end{align}
The smaller $\alpha$, the better the approximation of the hard
constraint~(\ref{eq:constraint}) at the expense of a more
constrained optimization problem. This problem is obtained as a
special case of
Problem~\eqref{eq:objective-2}--\eqref{eq:constraint-2} by setting
$R=1$ and defining $\phi$ as
$$\phi_{CC}(\eta)=1-\prod_{i=1}^r \indic{]-\infty, 0]}(\eta_i)-\alpha,$$
where $\indic{]-\infty, 0]}(\cdot)$ is the standard indicator
function. We now study the convexity of \eqref{e:constr} with
respect to $\theta$.

\subsection{The fixed feedback case}
\label{sec:openloop}

First assume that the feedback term $\bar G$ in~\eqref{e:ubardef}
is fixed and consider the convexity of the optimization
problem~\eqref{e:objectiveCC}--\eqref{e:constr} with respect to
the open loop control action $\bar d$. That is, for a given $\bar
G$ in the form~\eqref{e:parstruc}, the parameter space $\Theta$
becomes the set $\{(\bar G,\bar d),~\forall \bar d\in\RR^{Nm}\}$.
For the given $\bar{G}$ and $i=1,\ldots, r$ define
$$g_i(\bar d,\bar w)=\eta_i(\bar{A}\x0+\bar{B}(\bar{G}\bar{w}+\bar{d})+\bar{D}\bar{w},\bar{G}\bar{w}+\bar{d}),$$
where $\eta_i:\RR^{(N+1)n + Nm} \to \RR$ is the $i$-th element of
the constraint function $\eta$. Define
\begin{equation}
\label{eq:reducedprob}
p(\bar{d})=\PP[g_1(\bar d, \bar w)\leq 0,\ldots
g_r(\bar d, \bar w)\leq 0]
\end{equation}
and $\F_{CC}\triangleq\{\bar{d}:~p(\bar{d})\geq 1-\alpha\}$.
Observe that $\F_{CC}$ corresponds to the constraint set dictated
by~\eqref{e:dynconstr}--\eqref{e:constr} when $\bar G$ is fixed.

\begin{proposition}
\label{thm:convexopenloop} Assume that $\bar w$ has a continuous
distribution with log-concave probability density and that, for
$i=1,\ldots, r$, $g_i:\RR^{Nm+Nn}\to \RR$ is quasi-convex. Then,
for any value of $\alpha\in ]\,0,1[$, $\F_{CC}$ is convex. As a
consequence, under Assumption~\ref{a:convexobj} and for any
$\alpha\in ]\,0,1[$, the optimization problem
\begin{align*}
\inf_{\theta\in \Theta}\quad&\EE[V(\bar{x}_\theta,\bar{u}_\theta)]\\
\mathrm{subject~to}\quad&\eqref{e:ubardef},~\eqref{eq:cloopsys}~\mathrm{and}~p(\bar{d})\geq 1-\alpha
\end{align*}
is convex.
\end{proposition}
\begin{proof}
It follows from~\cite[Theorem 10.2.1]{prekopa}
that~\eqref{eq:reducedprob} is a log-concave function of
$\bar{d}$, i.e. the mapping $\bar{d}\mapsto \log p(\bar{d})$ is
concave. Since $\log$ is a monotone increasing function, we may
write that $\F_{CC}=\{\bar{d}:~\log p(\bar{d})\geq \log
(1-\alpha)\}$. Hence, $\F_{CC}$ is a convex set. The convexity of
the optimization problem follows readily from
Assumption~\ref{a:convexobj}.
\end{proof}%
Among others, Gaussian, exponential and uniform distributions are
continuous with log-concave probability density. As for the
functions $g_i$, one case of interest where the assumptions of
Proposition~\ref{thm:convexopenloop} are fulfilled is when
$\eta(\bar{x},\bar{u})$ is affine in $\bar{x}$ and $\bar{u}$. This
is the case of polytopic constraints, which will be treated
extensively in the next section. Apparently, this convexity result
cannot be applied to the ellipsoidal constraints treated
subsequently in Section~\ref{sec:ellipconstr}, nor can it be
extended to the general constraint~\eqref{e:constr} with both
$\bar d$ and $\bar G$ varying. Loosely speaking, the latter is
because the functions $g_i$ are not simultaneously quasi-concave
in $\bar{w}$ and $\bar{G}$. In the next sections we will develop
convex conservative approximations of~\eqref{e:constr} for various
definitions of $\eta$.

\subsection{Polytopic Constraint Functions}
\label{s:polyt}

Throughout the rest of Section~\ref{sec:cc} we shall rely on the
following assumption.
\begin{assumption}
\label{a:Gaussian} $\bar w$ is a Gaussian random vector with mean
zero and covariance matrix $\bar \Sigma>0$, denoted
by~$\bar{w}\sim\mc N(0,\bar\Sigma)$.
\end{assumption}%
Polytopic constraint functions
\begin{equation}
\label{e:etadef}
    \eta(\bar x_\theta,\bar u_\theta)=T^x\bar x_\theta +T^u \bar u_\theta - y,
\end{equation}
where $T^x\in \mb R^{r\times (N+1)n}, T^u\in \mb R^{r\times Nm}$,
and $y\in \mb R^r$, describe one of the most common types of
constraints. In light of~\eqref{e:ubardef}
and~\eqref{eq:cloopsys},
\begin{equation}
\label{e:appr2}
    \eta(\bar x_\theta,\bar u_\theta) =  h_\theta + P_\theta\bar w,
\end{equation}
where $h_\theta = (T^x\bar Ax_0- y)+(T^x \bar B + T^u)\bar d,$ and
$P_\theta = (T^x\bar D + (T^x\bar B + T^u)\bar G)$. It is thus
apparent that $\eta(\bar{x}_\theta,\bar{u}_\theta)$ is affine in
the parameters $\theta$. Yet, in general,
constraint~\eqref{e:constr} is nonconvex.
We now describe three approaches to approximate
constraints~\eqref{e:constr} that lead to convex conservative
constraints.

\subsubsection{Approximation via constraint separation}
\label{s:sep} Constraint~\eqref{e:constr} requires us to satisfy
$\eta(\bar x_\theta, \bar u_\theta)\leq 0$ with probability of at
least $1-\alpha$. Here $\eta \in \mb R^r$ and the inequality $\eta
\leq 0$ is interpreted componentwise. One idea is to select
coefficients $\alpha_i\in (0,1)$ such that $\sum_{i=1}^r \alpha_i
= \alpha$ and to satisfy the inequalities $\mb P(\eta_i\leq 0)\geq
1-\alpha_i$, with $i=1,\ldots, r$.
Note that this choice is obtained in~\eqref{eq:constraint-2} by
setting $R=r$ and
$\phi_i(\eta)=1-\indic{]-\infty,0]}(\eta_i)-\alpha_i$ where, for
$i=1,\ldots, r$, $\phi_i:\RR^r\to\RR$ denotes the $i$-th component
of function $\phi$.
\par Let
$h_{i,\theta}$ and $P_{i,\theta}^T$ be the $i$-th entry of
$h_\theta$ and the $i$-th row of $P_\theta$, respectively, and let
$\bar\Sigma^{\frac{1}{2}}$ be a symmetric real matrix square root
of $\bar\Sigma$.
\begin{proposition}
\label{prop:constrsep} Let $\alpha_i\in ]0,1[$. Under
Assumption~\ref{a:Gaussian}, the constraint
    \[
        \mb P(\eta_i(\bar x_\theta, \bar u_\theta) \leq 0)\geq 1-\alpha_i,
    \]
    with $\eta$ defined as in~\eqref{e:etadef}, is equivalent to the second-order cone constraint in the parameters $\theta \in\Theta$
$$h_{i, \theta} + \beta_i\norm{\bar \Sigma^{\frac{1}{2}}P_{i,
\theta}}\leq 0$$ where $\beta_i= \sqrt 2\erf^{-1}(1-2\alpha_i)$
and $\erf^{-1}(\cdot)$ is the inverse of the standard error
function $\erf(x) = \frac{2}{\sqrt{\pi}}\int_{0}^x
\epower{-u^2}\:\mrm du$. As a consequence, under
Assumption~\ref{a:convexobj} and if
$\alpha_1+\ldots+\alpha_r=\alpha$, the problem
\begin{align*}
\inf_{\theta\in \Theta}\quad&\EE[V(\bar{x}_\theta,\bar{u}_\theta)] \\
\mathrm{subject~to}\quad&\eqref{e:ubardef},~\eqref{eq:cloopsys}~\mathrm{and}
~h_{i, \theta} + \beta_i\norm{\bar \Sigma^{\frac{1}{2}}P_{i,
\theta}}\leq 0,~ i=1,\ldots, r
\end{align*}
is a convex conservative approximation of
Problem~\eqref{e:objectiveCC}--\eqref{e:constr}.
\end{proposition}
\begin{proof}
    From Eq.~\eqref{e:appr2} we can write
    $\eta_i = h_{i, \theta} + P_{i, \theta}^T \bar w$.
    Since $\bar w\sim\mathcal N(0,\bar\Sigma)$, $\eta_i$ is also Gaussian with distribution $\mathcal N(h_{i, \theta}, P_{i, \theta}^T\bar\Sigma P_{i,\theta})$.
    It is easily seen that, for any scalar Gaussian random variable $X$ with distribution $\mathcal
    N(\mu,\sigma^2)$,
    \[
        \mb P(X\leq0)\geq 1-\alpha  \iff \mu+\beta\sigma\leq0,
    \]
    where $\beta = \sqrt 2\erf^{-1}(1-2\alpha)$.
    Hence the constraint $\mb P(\eta_i(\bar x_\theta, \bar u_\theta)\leq0)\geq 1-\alpha_i$ is equivalent to $h_{i, \theta} + \beta_i\norm{\bar \Sigma^{\frac{1}{2}}P_{i, \theta}}\leq 0$,
    where $\beta_i=\sqrt 2\erf^{-1}(1-2\alpha_i)$. Since $h_{i, \theta}$ and $P_{i, \theta}$ are both affine in the parameters $\theta=(\bar{G},\bar{d})$,
    the above constraint is a second-order cone constraint in
    $\theta$. It is easy to see
that this choice guarantees that $\mb P(\eta(\bar x_\theta, \bar
u_\theta) \le 0) \ge 1-\alpha$.
\end{proof}

\subsubsection{Approximation via confidence ellipsoids}
\label{s:confellipsoids}

The approach of Section~\ref{s:sep} may be too conservative since
the probability of a union of events is approximated by the sum of
the probabilities of the individual events. One can also calculate
a conservative approximation of the union at once.
Constraint~\eqref{e:constr} with $\eta(\bar x_\theta, \bar
u_\theta)$ as in~\eqref{e:appr2} restricts the choice of
$P_\theta$ and $h_\theta$ to be such that, with a probability of
$1-\alpha$ or more, the realization of random vector $\eta(\bar
x_\theta, \bar u_\theta)$ lies in the negative orthant $\eta(\bar
x_\theta, \bar u_\theta)\leq 0$. In general, it is difficult to
describe this constraint explicitly since it involves the
integration of the probability density of $\eta$ over the negative
orthant. However, an explicit approximation of the constraint can
be computed
by ensuring that the $100(1-\alpha)\%$ confidence ellipsoid of
$\eta$ is contained in the negative orthant. Fulfilling this
requirement automatically implies that the probability of
$\eta(\bar x_\theta, \bar u_\theta)\leq 0$ is strictly greater
than $1-\alpha$.\par Since $\bar w\sim\mc N(0,\bar\Sigma)$, it
follows that $\eta(\bar x_\theta, \bar u_\theta) =
h_\theta+P_\theta\bar w \sim \mc N(h_\theta,\bar\Sigma_{\theta})$,
with $\bar\Sigma_{\theta} := P_\theta\bar\Sigma P_\theta^T$.
Consider the case where $\bar\Sigma_\theta$ is invertible. Define
the $r$-dimensional ellipsoid
\begin{equation}
\label{e:ellipsoid}
    \mc E(h_\theta,\bar \Sigma_\theta, \beta) = \bigl\{\eta\in\RR^r: \;\big|\;(\eta-h_\theta)^T \bar \Sigma_\theta^{-1}(\eta-h_\theta)\leq \beta^2\bigr\},
\end{equation}
where $\beta>0$ is a parameter specifying the size of the
ellipsoid. Notice that, in general, $\bar\Sigma_\theta$ is
invertible when $r\leq Nn$ (i.e. the number of constraints is less
than the total dimension of the process noise). If $r > Nn$, as
there are $Nn$ independent random variables in the optimization
problem, the following result still holds with $Nn$ in place of
$r$.

\begin{proposition}
\label{prop:confellip}
    Let $\alpha\in\;]0, 1[$. Under Assumption~\ref{a:Gaussian},
    the constraint $\mb P[\eta(\bar x_\theta, \bar
u_\theta) \le 0] \ge 1-\alpha$ with $\eta$ defined as
in~\eqref{e:etadef} is conservatively approximated by the
constraint
    \begin{equation}
    \label{eq:confellip}
\mc E\big(h_\theta,\bar \Sigma_\theta, \beta(\alpha) \big)\subset (-\infty, 0\,]\,^r,
\end{equation}
where $\beta(\alpha)=\sqrt{F^{-1}(1-\alpha)}$ and $F(\cdot)$ is
the probability distribution function of a $\chi^2$ random
variable with $r$ degrees of freedom.
Moreover,~\eqref{eq:confellip} can be reformulated as the set of
second-order cone constraints
\begin{equation}
\label{eq:equivconfellip}
h_{i,\theta} + \beta(\alpha) \norm{\bar\Sigma^\frac{1}{2}P_{i,
\theta}} \leq 0,\quad i=1, \ldots,
 r.
 \end{equation}
As a consequence, under Assumption~\ref{a:convexobj}, the problem
\begin{align*}
\inf_{\theta\in \Theta}\quad&\EE[V(\bar{x}_\theta,\bar{u}_\theta)] \\
\mathrm{subject~to}\quad&\eqref{e:ubardef},~\eqref{eq:cloopsys}~\mathrm{and}~\eqref{eq:equivconfellip}
\end{align*}
is a convex conservative approximation of
Problem~\eqref{e:objectiveCC}--\eqref{e:constr}.
\end{proposition}
\begin{proof}
Since $\eta(\bar x_\theta, \bar u_\theta)\sim \mc
N(h_\theta,\bar\Sigma_{\theta})$, the random variable
$$\big(\eta(\bar x_\theta, \bar u_\theta)-h_\theta\big)^T \bar
\Sigma_\theta^{-1}\big(\eta(\bar x_\theta, \bar
u_\theta)-h_\theta\big)$$ is $\chi^2$ with $r$ degrees of freedom.
Then, choosing $\beta$ such that $F(\beta^2)=1-\alpha$ guarantees
that $\mc E\big(h_\theta,\bar \Sigma_\theta, \beta(\alpha) \big)$
is the $100(1-\alpha)\%$ confidence ellipsoid for
$\eta(\bar{x}_\theta,\bar{u}_\theta)$. Finally,
under~\eqref{eq:confellip}, $\mb P[\eta(\bar x_\theta, \bar
u_\theta) \le 0] \ge \PP[\eta(\bar x_\theta, \bar u_\theta)\in\mc
E\big(h_\theta,\bar \Sigma_\theta, \beta(\alpha) \big)]=1-\alpha$,
which proves the first claim. To prove the second claim, note
that~\eqref{e:ellipsoid} can alternatively be represented as
\begin{equation}
\label{e:arbit1}
\mc E\big(h_\theta,\bar \Sigma_\theta, \beta \big)=
    \bigl\{\eta\in\RR^r \:\big|\: \eta =  h_\theta+M_\theta u ,\; \norm u\leq 1\bigr\},
\end{equation}
where $M_\theta = \beta \bar\Sigma_{\theta}^\frac{1}{2}$. Since
$\eta\leq 0$ if and only if $e_i^T \eta\leq 0,\: i =1,\cdots,r$,
where the $e_i$ denote the standard basis vectors in $\RR^r$, we
may rewrite~\eqref{eq:confellip} as
\[
    e_i^T(h_\theta + M_\theta u)\leq 0 \quad \forall \norm{u} \leq 1,
\]
or equivalently
\begin{align*}
    \sup_{\norm u\leq 1} e_i^T(h_\theta + M_\theta u) \leq 0
\end{align*}
for $i =1,\cdots,r$. For each $i$, the supremum is attained with
$u = M_\theta^T e_i/\norm{M_\theta^T e_i}$; therefore, the above
is equivalent to $e_i^T h_\theta +\norm{M_\theta^Te_i} \leq 0$.
Clearly, $e_i^T h_\theta=h_{i,\theta}$. Moreover, since $M_\theta
= \beta \bar\Sigma_{\theta}^\frac{1}{2}$, we have
\[
    \norm{M_\theta^Te_i} = \beta \sqrt {e_i^T\bar\Sigma_{\theta}e_i} = \beta \sqrt{ e_i^T P_\theta\bar\Sigma P_\theta^T e_i } =
\beta\norm{\bar\Sigma^\frac{1}{2}P_{i, \theta}}.
\]
Therefore, constraint~\eqref{eq:confellip} reduces
to~\eqref{eq:equivconfellip}. Since the variables $h_{i,\theta}$
and $P_{i,\theta}$ are affine in the original parameters $\theta =
(\bar G, \bar d)$, this is an intersection of second order cone
constraints. As a result, under the additional
Assumption~\ref{a:convexobj}, the optimization of
$\EE[V(\bar{x}_\theta,\bar{u}_\theta)]$
with~\eqref{eq:equivconfellip} in place of~\eqref{e:constr} is a
convex conservative approximation
of~\eqref{e:objectiveCC}--\eqref{e:constr}.
\end{proof}

\subsubsection{Comparison of constraint separation and confidence ellipsoid methods}
\label{s:comparison}

In light of Propositions~\ref{prop:constrsep}
and~\ref{prop:confellip}, both approaches lead to formulating
constraints of the form
$$\mu_i+\beta\sigma_i\leq 0,$$
with $i=1,\ldots, r$, where $\mu_i$ and $\sigma_i$ are the mean
and the standard deviation of the scalar random variables
$\eta_i(\bar{x}_\theta,\bar{u}_\theta)$, and smaller values of the
constant $\beta$ correspond to less conservative approximations of
the original chance constraint. For a given value of $\alpha$, the
value of $\beta$ depends on the number of constraints $r$ in a way
that differs in the two cases. In the confidence ellipsoid method,
in particular, the value of $\beta$ is determined by $N n$ (the
total dimension of the process noise) when $r\geq N n$. In
Figure~\ref{figure:betagrowth}, we compare the growth of $\beta$
in the two approaches under the assumption that $r$ grows linearly
with the horizon length $N$. (For the constraint separation method
we choose $\alpha_1=\ldots=\alpha_r=\alpha/r$, so that the value
of $\beta$ is the same for all constraints.)
\begin{figure}
\centering
\includegraphics[keepaspectratio=true,width=11cm]{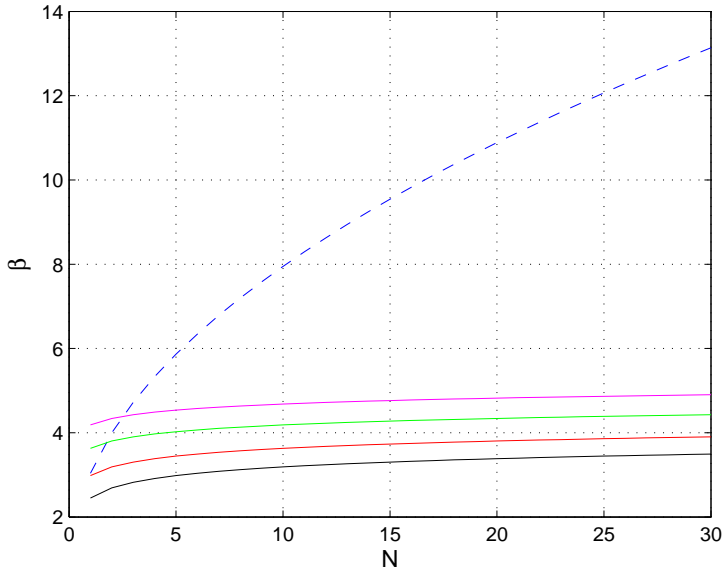}
\caption{Values of $\beta$ as a function of the horizon length $N$ for the ellipsoidal method (dashed line) and the constraint separation method (solid lines). It is assumed that
$r=N\cdot\bar{r}\cdot(n+m)$, where $n=5$ is the dimension of the system state and of the process noise, $m=2$ is the dimension of the input $u$. $\bar r$ reflects the number of constraints per stage;
we take $\bar{r}=2,10,100,1000$.
The value of $\beta$ is invariant with respect to $\bar{r}$ (single dashed line) for the ellipsoidal method, while it increases with $\bar{r}$ (multiple solid lines) for the constraint separation method.
}\label{figure:betagrowth}
\end{figure}
The increase of $\beta$ is quite rapid in the confidence ellipsoid
method, which is only effective for a small number of constraints.
An explanation of this phenomenon is provided by the following
fact, that is better known by the name of (classical)
``concentration of measure'' inequalities; proofs may be found in,
e.g.,~\cite{ref:bogachev}.
    \begin{proposition}
    \label{p:conc}
        Let $\Gamma_{h, \Sigma'}$ be the $r$-dimensional Gaussian measure with mean $h$ and (nonsingular) covariance $\Sigma'$, i.e.,
        $$ \Gamma_{h, \Sigma'}(\mrm d\xi)= \frac{1}{(2\pi)^{r/2}\sqrt{\det\Sigma'}}\exp\biggl(-\frac{1}{2}\inprod{\xi-h}{\Sigma'^{-1}(\xi-h)}\biggr) \mrm d\xi.
        $$
        Then for $\eps\in\;]0, 1[$,
        \begin{enumerate}[label=(\roman*), leftmargin=*, widest=iii]
            \item $\displaystyle{\Gamma_{h, \Sigma'}\biggl(\biggl\{\xi\;\bigg|\norm{\xi - h}_{\Sigma'^{-1}} > \sqrt{\frac{r}{1-\eps}}\biggr\}\biggr) \le \epower{-\frac{r\eps^2}{4}}}$;
            \item $\displaystyle{\Gamma_{h, \Sigma'}\bigl(\bigl\{\xi\big|\norm{\xi - h}_{\Sigma'^{-1}} < \sqrt{r(1-\eps)}\bigr\}\bigr) \le \epower{-\frac{r\eps^2}{4}}}$.
        \end{enumerate}
    \end{proposition}

The above proposition states that as the dimension $r$ of the
Gaussian measure increases, its mass concentrates in an
ellipsoidal shell of `mean-size' $\sqrt r$. It readily follows
that since $\eta(\bar x_\theta, \bar u_\theta)$ is a
$r$-dimensional Gaussian random vector, its mass concentrates
around a shell of size $\sqrt{r}$. Note that the bounds
corresponding to (i) and (ii) of Proposition~\ref{p:conc} in the
case of $\eta(\bar x_\theta, \bar u_\theta)$ are independent of
the optimization parameters $\theta$; of course the relative sizes
of the confidence ellipsoids change with $\theta$ (because the
mean and the covariance of $\eta(\bar x_\theta, \bar u_\theta)$
depend on $\theta$), but Proposition~\ref{p:conc} shows that the
size of the confidence ellipsoids grow quite rapidly with the
dimension of the noise and the length of the optimization horizon.
Intuitively one would expect the ellipsoidal constraint
approximation method to be more effective than the cruder
approximation by constraint separation.
Figure~\ref{figure:betagrowth} and Proposition~\ref{p:conc}
however suggest that this is not the case in general; for large
numbers of constraints (e.g. longer MPC prediction horizon) the
constraint separation method is the less conservative.

\subsubsection{Approximation via expectations}
\label{section:approxviaexp}
For any $r$-dimensional random vector
$\eta$, we have
$1-\mb P\bigl(\eta \le 0\bigr) = \mb
E\left[1-\prod_{i=1}^{r}\indic{]-\infty, 0]}(\eta_i)\right]$.
Using this fact one can arrive at conservative convex
approximations of the chance-constraint~(\ref{e:constr}) by
replacing the function in the expectation with appropriate
approximating functions. For $t_i>0$, $i=1,\ldots, r$, consider
$$
    \varphi(\eta) = \sum_{i=1}^{r}{\exp(t_i\eta_i)}.
$$
\begin{lemma}
\label{p:expcons} For any $r$-dimensional random vector $\eta$,
$\EE[\varphi(\eta)]\geq 1- \PP[\eta\leq 0]$.
\end{lemma}
\begin{proof}
For every fixed value of $\eta$, it holds that $\varphi(\eta)\geq
1-\prod_{i=1}^{r}\indic{]-\infty, 0]}(\eta_i)$. Hence
$\EE[\varphi(\eta)]\geq \EE[1-\prod_{i=1}^{r}\indic{]-\infty,
0]}(\eta_i)] =1- \PP[\eta\leq 0]$.
\end{proof}

\begin{proposition}
\label{proposition:GaussExp} Under Assumption~\ref{a:Gaussian},
for $\eta$ defined as in~\eqref{e:etadef}, it holds that
$$\EE\big[\varphi\big(\eta(\bar{x}_\theta,\bar{u}_\theta)\big)\big]= \sum_{i=1}^r \exp \Big(t_ih_{i,\theta} +\frac{t_i^2}{2} ||\bar\Sigma^{\frac{1}{2}}P_{i,\theta}||^2\Big).$$
As a consequence, under Assumption~\ref{a:convexobj} and for any
choice of $t_i>0$, $i=1,\ldots, r$, the problem
\begin{align*}
\inf_{\theta\in \Theta}\quad&\EE[V(\bar{x}_\theta,\bar{u}_\theta)] \\
\mathrm{subject~to}\quad&\eqref{e:ubardef},~\eqref{eq:cloopsys}~\mathrm{and}~\sum_{i=1}^r \exp \Big(t_ih_{i,\theta} +\frac{t_i^2}{2} ||\bar\Sigma^{\frac{1}{2}}P_{i,\theta}||^2\Big)\leq \alpha
\end{align*}
is a convex conservative approximation of
Problem~\eqref{e:objectiveCC}--\eqref{e:constr}.
\end{proposition}
\begin{proof}
It is easily seen that, for any $r$-dimensional Gaussian random
vector $\eta$ with mean $\mu$ and covariance matrix $\Sigma'$, and
any vector $c\in\RR^r$, $\mb E\bigl[\exp(c^T\eta)\bigr] =
\exp\bigl(c^T\mu+\frac{1}{2}c^T\Sigma' c\bigr)$. Let us now write
$\eta_\theta$ in place of $\eta(\bar x_\theta, \bar u_\theta)$ for
shortness. By the hypotheses on $\bar{w}$, in the light
of~\eqref{e:appr2}, $\eta_\theta$ is Gaussian with mean $h_\theta$
and covariance $P_\theta\bar\Sigma P_\theta^T$. Then, for a vector
$c_i\in\RR^r$ with zero entries except for a coefficient $t_i$ in
the $i$-th position,
$$\EE[\exp(c_i^T\eta_\theta)]=\exp\Big(c_i^T h_\theta+\frac{1}{2}c_i^T P_\theta\bar\Sigma P_\theta^T c_i\Big)=
\exp\Big(t_ih_{i,\theta}+\frac{t_i^2}{2}P_{i,\theta}^T\bar\Sigma
P_{i,\theta} \Big).$$ Summing up for $i=1,\ldots, r$ yields the
first result. In order to prove the second statement, note that
\begin{align*}
 \EE\bigl[\varphi(\eta_\theta)\bigr] \leq \alpha\iff &\log\EE\bigl[\varphi(\eta_\theta)\bigr] \leq \log \alpha \\ \iff &
\log\Biggl(\sum_{i=1}^{r} {\exp\Bigl(t_ih_{i,\theta}+\frac{t_i^2}{2}\norm{\bar\Sigma^{\frac{1}{2}}P_{i,\theta}}^2\Bigr)}\Biggr) \leq \log(\alpha).
\end{align*}
For each $i$,
$t_ih_{i,\theta}+\frac{t_i^2}{2}\norm{\bar\Sigma^{\frac{1}{2}}P_{i,\theta}}^2$
is a convex function of the optimization parameters $\theta$.
By~\cite[Example~3.14]{ref:boyd04}, given convex functions
$g_1(\theta),\ldots,g_r(\theta)$, the function $f(\theta) =
\log\bigl(e^{g_1(\theta)}+\cdots+e^{g_k(\theta)}\bigr)$ is itself
convex. It follows that $\log \EE\bigl[\varphi(\eta_\theta)\bigr]$
is a convex function of $\theta$ and that the constraint set
$$\{\theta\in\Theta:~\log \EE[\varphi(\eta_\theta)]\leq \log
\alpha\}=\{\theta\in\Theta:~\EE[\varphi(\eta_\theta)]\leq
\alpha\}$$ is convex. Finally, from Lemma~\ref{p:expcons}, if
$\EE[\varphi(\eta_\theta)]\leq \alpha$ then $\PP[\eta_\theta\leq
0]\geq  1- \EE[\varphi(\eta_\theta)]\geq 1-\alpha$. Together with
Assumption~\ref{a:convexobj}, this implies that the optimization
problem with constraint $\sum_{i=1}^r \exp \Big(t_ih_{i,\theta}
+\frac{t_i^2}{2}
||\bar\Sigma^{\frac{1}{2}}P_{i,\theta}||^2\Big)\leq \alpha$ in
place of~\eqref{e:constr} is a convex conservative approximation
of~\eqref{e:objectiveCC}--\eqref{e:constr}.
%
%
\end{proof}

This convex approximation of
Problem~\eqref{e:objectiveCC}--\eqref{e:constr} is obtained
in~\eqref{eq:objective-2}--\eqref{eq:constraint-2} by setting
$R=1$ and $\phi(\eta)=\varphi(\eta)-\alpha$. The result that the
approximation is conservative relies essentially on the fact that,
with this choice, $\phi(\eta)\geq \phi_{CC}(\eta)$ $\forall
\eta\in\RR^r$ (see Lemma \ref{p:expcons}). This result can be
generalized: Given any two functions $\phi',\phi'':\RR^r\to \RR$
such that $\phi'(\eta)\geq \phi''(\eta)$ $\forall \eta\in\RR^r$,
constraint~\eqref{eq:constraint-2} with $\phi=\phi'$ is more
conservative than the same constraint with $\phi=\phi''$. This
type of analysis can be exploited to compare different
probabilistic constraints and to minimize the conservatism of the
convex approximations with respect with the tunable parameters,
but is not fully pursued here.

\subsection{Ellipsoidal Constraint Functions}
\label{sec:ellipconstr}
Consider the constraint function
$$
    \eta(\bar x_\theta, \bar u_\theta) = \left(\begin{bmatrix}
        \bar x_\theta \\ \bar u_\theta
        \end{bmatrix}-\delta \right)^T \Xi\left(\begin{bmatrix}
        \bar x_\theta \\ \bar u_\theta
        \end{bmatrix}-\delta \right)-1,
$$
where $\Xi \geq 0$ and $\delta$ are given. Then the constraint
$\eta(\bar x_\theta, \bar u_\theta)\leq 0$ restricts the vector
$\begin{bmatrix} \bar x_\theta^T & \bar u_\theta^T
\end{bmatrix}^T$ to an ellipsoid with center $\delta$ and shape
determined by $\Xi$.
We now provide an approximation of the chance
constraint~\eqref{e:constr} that is a semi-definite program in the
optimization parameters $\theta = (\bar G, \bar d)$. Similar to
\S\ref{s:confellipsoids}, the idea is to ensure that the
$100(1-\alpha)\%$ confidence ellipsoid of $\bar w$ is such
that~\eqref{e:constr} holds. To this end, let
$$
    y_\theta = \begin{bmatrix}
        \bar x_\theta \\ \bar u_\theta
        \end{bmatrix} -\delta  = \begin{bmatrix}
        \bar Ax_0+\bar B(\bar G \bar w + \bar d ) + \bar D \bar w \\
         \bar G \bar w + \bar d
        \end{bmatrix} - \delta = h_\theta' + P_\theta' \bar w,
$$
where $h_\theta' = \begin{bmatrix}
\bar Ax_0+\bar B\bar d \\
 \bar d
\end{bmatrix} - \delta$ and $P'_\theta = \begin{bmatrix}
\bar B\bar G + \bar D \\
 \bar G
\end{bmatrix} $. 

\begin{proposition}
\label{prop:ellipLMI} Define
$S_\theta=\beta(\alpha)\Xi^{\frac{1}{2}}P'_\theta\bar{\Sigma}^{\frac{1}{2}}$,
with $\beta(\alpha)$ as in Proposition~\ref{prop:confellip}, and
$\xi_\theta = \Xi^{\frac{1}{2}}h_\theta'$. Then
\begin{equation}
\label{e:bigge}
    \begin{bmatrix}
        -\lambda + 1 & 0 & \xi_\theta^T\\
        0 & \lambda I & (S_\theta)^T\\
        \xi_\theta & S_\theta & I
    \end{bmatrix} \ge 0,\qquad \lambda>0,
\end{equation}
is a Linear Matrix Inequality (LMI) in the unknown parameters
$\theta$ and $\lambda$. If $(\theta,\lambda)$ is a solution
of~\eqref{e:bigge}, then $\theta$ satisfies~\eqref{e:constr}. As a
consequence, under Assumption~\ref{a:convexobj}, the problem
\begin{align*}
\inf_{\theta\in \Theta,\lambda}\quad&\EE[V(\bar{x}_\theta,\bar{u}_\theta)] \\
\mathrm{subject~to}\quad&\eqref{e:ubardef},~\eqref{eq:cloopsys}~\mathrm{and}~\eqref{e:bigge}
\end{align*}
is a convex conservative approximation of
Problem~\eqref{e:objectiveCC}--\eqref{e:constr}.
\end{proposition}
\begin{proof}
The inequality~\eqref{e:constr} may be equivalently represented as
\begin{align}
\label{e:ellips2}
    \mb P\bigl(y_\theta^T\Xi y_\theta & - 1 \leq 0\bigr) \geq 1-\alpha\nonumber\\
    \iff & \;\; \mb P\bigl(\|\Xi^\frac{1}{2}y_\theta\|^2-1\leq0\bigr) \geq 1-\alpha\nonumber\\
    \iff & \;\; \mb P\bigl(\|\Xi^\frac{1}{2}(h'_\theta + P'_\theta\bar w)\|^2-1\leq0\bigr) \geq 1-\alpha\nonumber\\
    \iff & \;\; \mb P\bigl(\|\xi_\theta + S'_\theta\bar w\|^2-1\leq0\bigr) \geq 1-\alpha,
\end{align}
where $S'_\theta = \Xi^{\frac{1}{2}}P'_\theta$. Since
$\bar{w}\sim\mc N(0,\bar\Sigma)$, one can compute $\beta(\alpha)>
0$ such that $\norm {\bar\Sigma^{-1/2}\bar w}^2\leq
\beta(\alpha)^2$ specifies the required $100(1-\alpha)\%$
confidence ellipsoid of $\bar w$. Hence, we need to ensure that
$\norm{\bar\Sigma^{-1/2}\bar w}^2\leq \beta(\alpha)^2 \Rightarrow
\norm{\xi_\theta + S'_\theta\bar w}^2\leq1$. This is equivalent to
$$
\sup_{\norm{\bar\Sigma^{-1/2}\bar w}^2\leq \beta(\alpha)^2} \norm{\xi_\theta + S'_\theta\bar w}^2\leq 1\quad
    \iff \quad \sup_{\norm{\bar v}^2\leq 1} \norm{\xi_\theta+S_\theta\bar v}^2\leq 1.
$$
It follows from~\cite[p.~653]{ref:boyd04} that $\sup_{\norm{\bar
v} \le 1} \norm{\xi_\theta + S_\theta\bar v}^2 \le 1$ if and only
if there exists $\lambda \ge 0$ such that
\[
    \begin{bmatrix}
        -\xi_\theta^T \xi_\theta - \lambda + 1 & \xi_\theta^T S_\theta\\
        (S_\theta)^T \xi_\theta & \lambda I - (S_\theta)^T S_\theta
    \end{bmatrix} \ge 0.
\]
Using Schur complements the last relation can be rewritten
equivalently as~\eqref{e:bigge}. Therefore, any solution
of~\eqref{e:bigge} implies~\eqref{e:constr}. To verify
that~\eqref{e:bigge} is an LMI, note that $S_\theta$ and
$\xi_\theta$ are affine in the optimization variables. Together
with the assumed convexity of
$\EE[V(\bar{x}_\theta,\bar{u}_\theta)]$, the last statement of the
proposition follows.
\end{proof}


\section{Integrated chance constraints}
\label{sec:icc} In this section we focus on the problem
\begin{align}
\inf_{\theta\in \Theta}\quad&\EE[V(\bar{x}_\theta,\bar{u}_\theta)] \label{eq:objectiveICC} \\
\textrm{subject to}\quad&\eqref{e:ubardef},~\eqref{eq:cloopsys}\textrm{ and} \\
&J_i(\theta)\leq \beta_i,~i=1,\ldots, r \label{eq:genICC}
\end{align}
where, for $i=1,\ldots, r$,
$J_i(\theta)\triangleq \EE[\varphi_i \circ
\eta_i(\bar{x}_\theta,\bar{u}_\theta)]$, functions
$\varphi_i:\RR\to \RR$ are measurable and the $\beta_i>0$ are
fixed
parameters. 
This problem corresponds to
Problem~\eqref{eq:objective-2}--\eqref{eq:constraint-2} when
setting $R=r$ and $\phi_i(\eta)=\varphi_i(\eta)-\beta_i$, with
$i=1,\ldots, r$. For the choice
\begin{equation}
\label{eq:ramp}
\varphi_i(z)=\begin{cases}
0,& \textrm{if }z\leq 0, \\
z, &\textrm{otherwise},
\end{cases}
\end{equation}
constraints of the form~(\ref{eq:genICC}) are known as integrated
chance constraints~\cite{haneveldICC,haneveldICC06}. In fact, one
may write (dropping the dependence of $\eta_i$ on $\bar{x}_\theta$
and $\bar{u}_\theta$ to simplify the notation)
$$
    \EE[\varphi_i(\eta_i)]=\EE\left[\int_{0}^{+\infty}\indic{[0,\eta_i )}(s) \,\mrm ds\right]=\int_{0}^{+\infty}\PP\bigl(\eta_i>s\bigr)\,\mrm ds,
$$
where $\indic{S}(\cdot)$ is the indicator function of set $S$ and
the second equality follows from Tonelli's theorem~\cite[Theorem
4.4.5]{dudley}. Therefore, constraint~\eqref{eq:genICC} is
equivalent to
\begin{equation}
\label{eq:icc}
    \int_0^{+\infty}\PP\bigl(\eta_i(\bar{x}_\theta,\bar{u}_\theta)>s\bigr)\,\mrm ds \leq \beta_i,
\end{equation}
whence the name integrated chance constraint. Note that
$\varphi_i$ plays the role of a penalty (or barrier) function that
penalizes violations of the inequality
$\eta_i(\bar{x},\bar{u})\leq 0$, and $\beta_i$ is a maximum
allowable cost in the sense of~(\ref{eq:icc}). Of course,
different choices of $\varphi_i$ will not guarantee the
equivalence between~(\ref{eq:genICC}) and~(\ref{eq:icc}). However,
they may be useful in deriving other quantitative chance
constraint-type approximations.

\subsection{Convexity of Integrated Chance Constraints}

We now establish sufficient conditions on the $\eta_i$ and
$\varphi_i$ for the convexity of the constraint set
\begin{equation}
\label{eq:Ficc}
\F_{\rm ICC}\triangleq \{\theta:~J_i(\theta)\leq \beta_i,~i=1,\ldots, r\}.
\end{equation}

The result is again a consequence of
Proposition~\ref{thm:phietaconvex}.

\begin{proposition}
\label{thm:etaphiconditions} Let the mappings
$\eta_i:\RR^{(N+1)n\times Nm}\to \RR$ be measurable and convex,
and let the $\varphi_i:\RR\to\RR$ be measurable, monotone
nondecreasing and convex. Assume that the $J_i(\theta)$ are finite
for all $\theta$. Then each $J_i(\theta)$ is a convex function of
$\theta$ and $\F_{\rm ICC}$ is a convex set. As a consequence,
under
Assumption~\ref{a:convexobj},~\eqref{eq:objectiveICC}--\eqref{eq:genICC}
is a convex optimization problem.
\end{proposition}
\begin{proof} Fix $\omega\in\Omega$ arbitrarily. Since the mapping $\theta\mapsto
\big(\bar{x}_\theta(\omega),\bar{u}_\theta(\omega)\big)$ is affine
and $(\bar{x},\bar{u})\mapsto \eta_i(\bar{x},\bar{u})$ is convex
by assumption, their composition $\theta\mapsto
\eta_i\big(\bar{x}(\omega),\bar{u}(\omega)\big)$ is a convex
function of $\theta$.
%
%
Using the assumption that $\varphi_i$ is monotone nondecreasing
and convex, we may apply Proposition~\ref{thm:phietaconvex} with
$\gamma(\omega,\theta)=\eta_i\big(\bar{x}_\theta(\omega),\bar{u}_{\theta}(\omega)\big)$
and $\varphi_i$ in place of $\varphi$ to conclude that
$J_i(\theta)=\EE[\varphi_i\circ\gamma(\omega,\theta)]$ is convex.
Hence, for any choice of $\beta_i$, the set $\F_{i}\triangleq
\{\theta:~J_i(\theta)\leq \beta_i\}$ is convex. Since $\F_{\rm
ICC}=\bigcap_{i=1}^r \F_i$, the convexity of $\F_{\rm ICC}$
follows. Together with Assumption~\ref{a:convexobj}, this proves
that~\eqref{eq:objectiveICC}--\eqref{eq:genICC} is a convex
optimization problem.
\end{proof}%
It is worth noting that the function $\varphi(\cdot)$ of
Section~\ref{section:approxviaexp} satisfies analogous
monotonicity and convexity assumptions with respect to each of the
$\eta_i$, with $i=1,\ldots, r$. Unlike those of
Section~\ref{sec:cc}, this convexity result is independent of the
probability distribution of $\bar{w}$. By virtue of the
alternative assumptions (iii$'$) and (iii$''$) of
Proposition~\ref{thm:phietaconvex}, the requirement that
$J(\theta)$ be finite for all $\theta$ may be relaxed. A
sufficient requirement is that there exist no two values $\theta$
and $\theta'$ such that $J(\theta)=+\infty$ and
$J(\theta')=-\infty$. In particular, provided measurable and
convex $\eta_i$, definition~(\ref{eq:ramp}) satisfies all the
requirements of Proposition~\ref{thm:etaphiconditions}.

\begin{example}\rm
The (scalar) polytopic constraint function:
$$\eta(\bar x_\theta,\bar u_\theta)=T_x\bar x_\theta +T_u \bar u_\theta-y$$
fulfills the hypotheses of Proposition~\ref{thm:etaphiconditions}. Hence, the corresponding integrated chance constraint is convex.
\end{example}

\begin{example}\rm
\label{ex:followon1} Following Example~\ref{thm:examples}, an
interesting case is that of ellipsoidal constraints. For an
$\big((N+1)n+Nm\big)$-size positive-semidefinite real matrix $\Xi$
and a vector $\delta\in\RR^{(N+1)n+Nm}$, define
\begin{equation}
\label{eq:ellipeta}
\eta(\bar{x}_\theta,\bar{u}_\theta)\triangleq \left(\begin{bmatrix} \bar{x}_\theta \\ \bar{u}_\theta\end{bmatrix}-\delta\right)^T\Xi
\left(\begin{bmatrix} \bar{x}_\theta \\ \bar{u}_\theta\end{bmatrix}-\delta\right)-1.
\end{equation}
This is a convex function of the vector $\begin{bmatrix}
\bar{x}_\theta^T  & \bar{u}_\theta^T\end{bmatrix}^T$ (it is the
composition of the convex mapping $\xi^T\Xi\xi$, $\Xi \geq 0$,
with the affine mapping $\xi=\begin{bmatrix} \bar{x}_\theta^T &
\bar{u}_\theta^T\end{bmatrix}^T-\delta$) and hence
Proposition~\ref{thm:etaphiconditions} applies. \end{example}
\begin{remark}\rm
A problem setting similar to Example~\ref{ex:followon1} with
quadratic expected-type cost function and ellipsoidal constraints
has been adopted in~\cite{primbs}, where hard constraints are
relaxed to expected-type constraints of the form
$\EE[\eta(\bar{x},\bar{u})]\leq \beta$. This formulation can be
seen as a special case of integrated chance constraints with
$\varphi(z)=z$ for all $z$. The choice of $\varphi$ within a large
class of functions is an extra degree of freedom provided by our
framework that may be exploited to establish tight bounds on the
probability of violating the original hard constraints, see
Lemma~\ref{p:expcons} for an example.
\end{remark}

\subsection{Numerical solution of optimization problems with ICC}

Even though ICC problems are convex in general, deriving efficient
algorithms to solve them is still a major
challenge~\cite{haneveldICC,haneveldICC06}. For certain ICCs, it
is possible however to derive explicit expressions for the
gradients of the constraint function. Provided the cost has a
simple (e.g. quadratic) form, this allows one to implement
standard algorithms (e.g. interior point
methods~\cite{ref:boyd04}) for the solution of the optimization
problem. \par Let $\bar{w}$ satisfy Assumption~\ref{a:Gaussian}
with, for simplicity, $\bar \Sigma$ equal to the identity matrix,
i.e. $\bar w\sim \mathcal N(0,I)$. Consider the problem with one
scalar constraint (the generalization to multiple (joint)
constraints is straightforward):
\begin{equation}
\label{eq:casestudy}
\begin{aligned}
\min_{\theta} \quad&\begin{bmatrix}\bar{x}_\theta^T & \bar{u}_\theta^T \end{bmatrix} M \begin{bmatrix}\bar{x}_\theta \\ \bar{u}_\theta \end{bmatrix} \\
\textrm{subject to}\quad &\eqref{e:ubardef},~\eqref{eq:cloopsys}\textrm{ and }\EE [\varphi(T_x\bar x_\theta +T_u \bar u_\theta-y)]\leq \beta,
 \end{aligned}
 \end{equation}
where $\varphi$ is defined as in~\eqref{eq:ramp}.
\begin{lemma}
Let $z$ be a Gaussian random variable with mean $\mu$ and variance
$\sigma^2>0$. Then
$$\mb E[\varphi(z)] = \sigma
g\Bigl(\frac{\mu}{\sigma}\Bigr)$$ where $$g(x)
=\frac{x}{2}\erfc\Bigl(\frac{-x}{\sqrt 2}\Bigr)+
\frac{1}{\sqrt{2\pi}}\exp \Bigl(\frac{-x^2}{2}\Bigr)$$ and
$\erfc(x) = 1 - \erf(x)$
is the standard complementary error function.
\end{lemma}
\begin{proof}
Since $z\sim \mathcal N(\mu, \sigma^2)$ it holds that $$\mb
E[\varphi(z)] =
\frac{1}{\sqrt{2\pi\sigma^2}}\int_{0}^{\infty}t\exp\left(-\frac{(t-\mu)^2}{2\sigma^2}\right)\;dt.$$
By the change of variable $y= \frac{t-\mu}{\sqrt{2}\sigma}$ one
gets
\begin{align*}
\mb E[\varphi(z)]=&\,\, \frac{1}{\sqrt{\pi}}\int_{\frac{-\mu}{\sqrt{2}\sigma}}^{\infty} (\mu+\sqrt{2}\sigma y)\exp(-y^2)dy \\
=&\,\,\frac{\mu}{\sqrt{\pi}}\int_{\frac{-\mu}{\sqrt{2}\sigma}}^{\infty} \exp(-y^2)dy + \sigma\sqrt{\frac{2}{\pi}} \int_{\frac{-\mu}{\sqrt{2}\sigma}}^{\infty} y\exp(-y^2)dy\\
=&\,\,\frac{\mu}{2}\erfc\left(\frac{-\mu}{\sqrt{2}\sigma}\right)+\frac{\sigma}{\sqrt{2\pi}}\exp\left(-\frac{\mu^2}{2\sigma^2}\right),
\end{align*}
which is equal to $\sigma g(\frac{\mu}{\sigma})$.
\end{proof}
Simple calculations and the application of this lemma yield the following result.
\begin{proposition}
Problem~\eqref{eq:casestudy} is equivalent to
\begin{equation}
\begin{aligned}
\textrm{minimize} \quad& h_1(\bar{d})+h_2(\bar{G}) \\
\textrm{subject to} \quad
&\sigma g\Bigl(\frac{\mu}{\sigma}\Bigr) \leq \beta
 \end{aligned}
 \end{equation}
 where
\begin{align*}
\mu =& (T^x\bar Ax_0-\bar y)+(T^x \bar B + T^u)\bar d, \\
\sigma =& ||(T^x\bar D + (T^x\bar B + T^u)\bar G)^T||_2, \\
h_1 (\bar d) =& \begin{bmatrix} (\bar{A}x_0+\bar{B}\bar d)^T & \bar d^T
\end{bmatrix}
M
\begin{bmatrix} \bar{A}x_0+\bar{B}\bar d \\ \bar d
\end{bmatrix}, \\
h_2(\bar G) =& \text{Tr}\Biggl(\begin{bmatrix} (\bar B \bar G + \bar D)^T & \bar G^T
\end{bmatrix}
M
\begin{bmatrix} \bar B \bar G + \bar D \\ \bar G
\end{bmatrix} \Biggr).
\end{align*}
\end{proposition}

Note that expectations have been integrated out. Now it is
possible to put the problem in a standard form for numerical
optimization. Let $\bar{G}_i$ and $\bar{D}_i$ be the $i$-th column
of $\bar{G}$ and $\bar{D}$, respectively. Redefine the
optimization variable $\theta$ as the vector $\bar{\theta}=
\begin{bmatrix}\bar d^T & \bar G_1^T & \cdots & \bar G_{Nn}^T
\end{bmatrix}^T$. Define $e$, $f$, $v$ and $L$ by
\begin{align*}
\mu = (T^x\bar Ax_0-\bar y)+(T^x \bar B + T^u)\bar d =&e+f^T\bar{\theta}, \\
 (T^x\bar D + (T^x\bar B + T^u)\bar G)^T =&v+L\bar{\theta}.
\end{align*}
\begin{corollary}
Problem~\eqref{eq:casestudy} is equivalent to
\begin{align*}
\text{minimize } \quad f_0(\bar{\theta})= &g_0(\bar{\theta}) + \sum_{i=1}^{Nn} g_i(\bar{\theta}) \\
\text{subject to } \quad H\bar{\theta} = &0, \\
f_1(\bar{\theta})\leq &0
\end{align*}
where the matrix $H$ in the equality constraint accounts for the
causal structure of $\bar{G}$, while
\begin{align*}
g_0(\bar{\theta})
=&\Bigl(\begin{bmatrix} \bar A x_0 \\ 0\end{bmatrix} + \begin{bmatrix} \bar B \\ I\end{bmatrix} H_0 \bar{\theta} \Bigr)^T M \Bigl(\begin{bmatrix} \bar A x_0 \\ 0\end{bmatrix} + \begin{bmatrix} \bar B \\ I\end{bmatrix} H_0 \bar{\theta} \Bigr), \\
g_i(\bar{\theta})
=&\Bigl(\begin{bmatrix} \bar D_i \\ 0\end{bmatrix} + \begin{bmatrix} \bar B \\ I\end{bmatrix} H_i \bar{\theta} \Bigr)^T M \Bigl(\begin{bmatrix} \bar D_i \\ 0\end{bmatrix} + \begin{bmatrix} \bar B \\ I\end{bmatrix} H_i \bar{\theta} \Bigr), \\
f_1(\bar{\theta}) =& \norm{v+L\bar{\theta}}g\Bigl(\frac{e+f^T\bar{\theta}}{\norm{v+L\bar{\theta}}}\Bigr)-\beta \leq 0,
\end{align*}
and $H_0$ and $H_i$ are selection matrices such that
$H_0\bar\theta=\bar d$ and $H_i\bar\theta=\bar G_i$.
\end{corollary}




We conclude the section by documenting the expressions of the
gradient and the Hessian of the constraint function $f_1(\bar
\theta)$.
\[\nabla f_1(\bar{\theta}) = \frac{1}{\sqrt{2\pi}}L^T \frac{v+L\bar{\theta}}{\norm{v+L\bar{\theta}}} \exp{\Bigl(-\frac{(e+f^T\bar{\theta})^2}{2\norm{v+L\bar{\theta}}^2}\Bigr)} + \frac{1}{2}\erfc{\Bigl(-\frac{e+f^T\bar{\theta}}{\sqrt{2}\norm{v+L\bar{\theta}}}\Bigr)}f\]
\\
\[\nabla^2 f_1(\bar{\theta}) = \frac{1}{\sqrt{2\pi}}\exp{\Bigl(-\frac{(e+f^T\bar{\theta})^2}{2\norm{v+L\bar{\theta}}^2}\Bigr)}\Bigl[J_1(\bar{\theta}) +J_2(\bar{\theta}) - J_3(\bar{\theta}) \Bigr]\]

where
\begin{align*}
J_1(\bar{\theta}) &= \frac{1}{\norm{v+L\bar{\theta}}}(L^TL+ff^T), \\
J_2(\bar{\theta}) &= \Bigl(\frac{(e+f^T\bar{\theta})^2 - \norm{v+L\bar{\theta}}^2}{\norm{v+L\bar{\theta}}^5}\Bigr) (L^T(v+L\bar{\theta})(v+L\bar{\theta})^TL), \\
J_3(\bar{\theta}) &= \Bigl(\frac{e+f^T\bar{\theta}}{\norm{v+L\bar{\theta}}^3}\Bigl) (L^T(v+L\bar{\theta})f^T + f(v+L\bar{\theta})^TL). \\
\end{align*}
The expressions of gradient and Hessian of the quadratic function
$f_0(\bar\theta)$, used e.g. by interior point method solvers, are
quite standard and will not be reported here.

\section{Simulation results}
\label{sec:num}

We illustrate some of our results with the help of a simple
example. Consider the mechanical system shown in
Figure~\ref{f:springs}.
\begin{figure}
{\small
\begin{center}
\begin{psfrags}
\psfrag{d1}[b]{$d_1$} \psfrag{d2}[b]{$d_2$} \psfrag{d3}[b]{$d_3$}
\psfrag{d4}[b]{$d_4$} \psfrag{m1}[]{$m_1$} \psfrag{m2}[]{$m_2$}
\psfrag{m3}[]{$m_3$} \psfrag{m4}[]{$m_4$} \psfrag{k1}[t]{$k_1$}
\psfrag{k2}[t]{$k_2$} \psfrag{k3}[t]{$k_3$} \psfrag{k4}[t]{$k_4$}
\psfrag{u1}[b]{$u_1$} \psfrag{u2}[b]{$u_2$} \psfrag{u3}[t]{$u_3$}
\includegraphics[width=.9\linewidth]{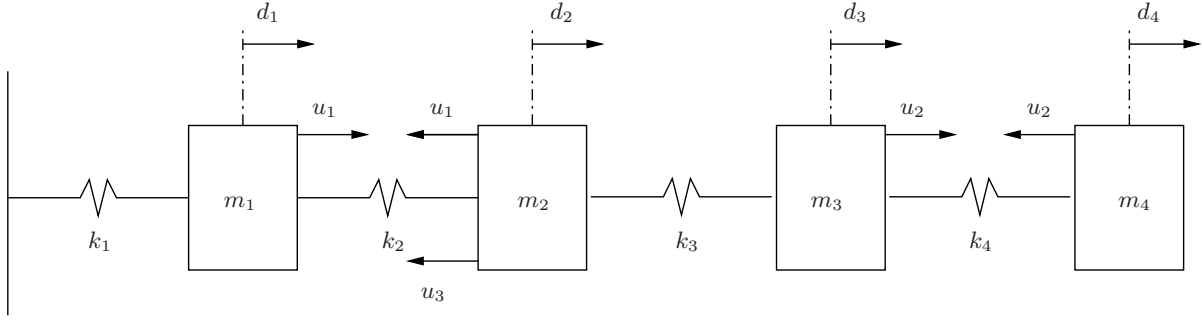}
\end{psfrags}
\end{center}
}
\caption{A mechanical system with springs and masses.}
\label{f:springs}
\end{figure}
$d_1, \cdots, d_4$ are displacements from an equilibrium position,
$u_1,\cdots, u_3$ are forces acting on the masses. In particular,
$u_1$ is a tension between the first and the second mass, $u_2$ is
a tension between the third and the fourth mass, and $u_3$ is a
force between the wall (at left) and the second mass. We assume
all mass and stiffness constants to be equal to unity, i.e. $m_1 =
\cdots = m_4 = 1, k_1 = \cdots = k_4 = 1$. We consider a
discrete-time model of this system with noise in the dynamics,
$$x(t+1) = A x(t) + B u(t) + w(t), $$
where $w$ is an i.i.d. noise process, $w(t) \sim
\mathcal{N}(0,\sigma^2I), \sigma = 0.05$ for all $t$, and
$x=\begin{bmatrix}d_1,d_2,d_3,d_4, \dot d_1, \dot d_2,\dot
d_3,\dot d_4
\end{bmatrix}^T$. The discrete-time dynamics are obtained
by uniform sampling of a continuous-time model at times $t\cdot
h$, with sampling time $h=1$ and $t=0,1,2,\ldots$, under the
assumption that the control action $u(t)$ is piecewise constant
over the sampling intervals $[t\cdot h,t\cdot h+h)$. Hence, $A =
e^{hA_c}$ and $B=A_c^{-1}(e^{hA_c}-I)B_c$, where $A_c$ and $B_c$,
defined as
\begin{align*}
A_c =&\,\, \begin{bmatrix}0_{4\times 4} & I_{4\times 4} \\ A_{c,21} & 0_{4 \times 4} \end{bmatrix}, &A_{c,21} =&\,\, \begin{bmatrix}-2& 1& 0& 0,\\
1& -2& 1& 0\\
0& 1& -2& 1\\
0& 0& 1& -1\\ \end{bmatrix},
\\
B_c = &\,\, \begin{bmatrix}0_{4\times 3} \\ B_{c,21} \end{bmatrix}, &B_{c,21}= &\,\,\begin{bmatrix} 1 & 0 & 0\\
-1& 0& -1\\
0& 1& 0\\
0& -1& 0\\  \end{bmatrix}
\end{align*}
are the state and input matrices of the standard ODE model of the
system.
\par
We are interested in computing the control policy that minimizes
the cost function
$$\EE\left[\sum_{t=0}^{N-1}\bigl(x(t)^TQx(t)+u(t)^TRu(t)\bigr)+x(N)^TQx(N)\right],
$$
where the horizon length is fixed to $N=5$, the weight matrices
are defined as $Q = \begin{bmatrix}I & 0 \\0 & 0
\end{bmatrix}$ (penalizing displacements but not their derivatives) and $R = I$.
The initial state is set to $x_0 = [0, 0, 0, 1, 0, 0, 0, 0]^T$. In
the absence of constraints, this is a finite horizon LQG problem
whose optimal solution is the linear time-varying feedback from
the state
$$u(t) =
-\big(B^TP(t+1)B+R\big)^{-1}B^TP(t+1)Ax(t),$$ where the matrices
$P(t)$ are computed by solving, for $t = N-1, \ldots, 0$, the
backward dynamic programming recursion
$$P(t)=Q+A^TP(t+1)A-A^TP(t+1)B(B^TP(t+1)B+R)^{-1}B^TP(t+1)A,$$ with $P(N)=Q$.
Simulated runs of the controlled system are shown in
Figure~\ref{f:LQG}.
\begin{figure}
\centering
\includegraphics[keepaspectratio=true,width=11cm]{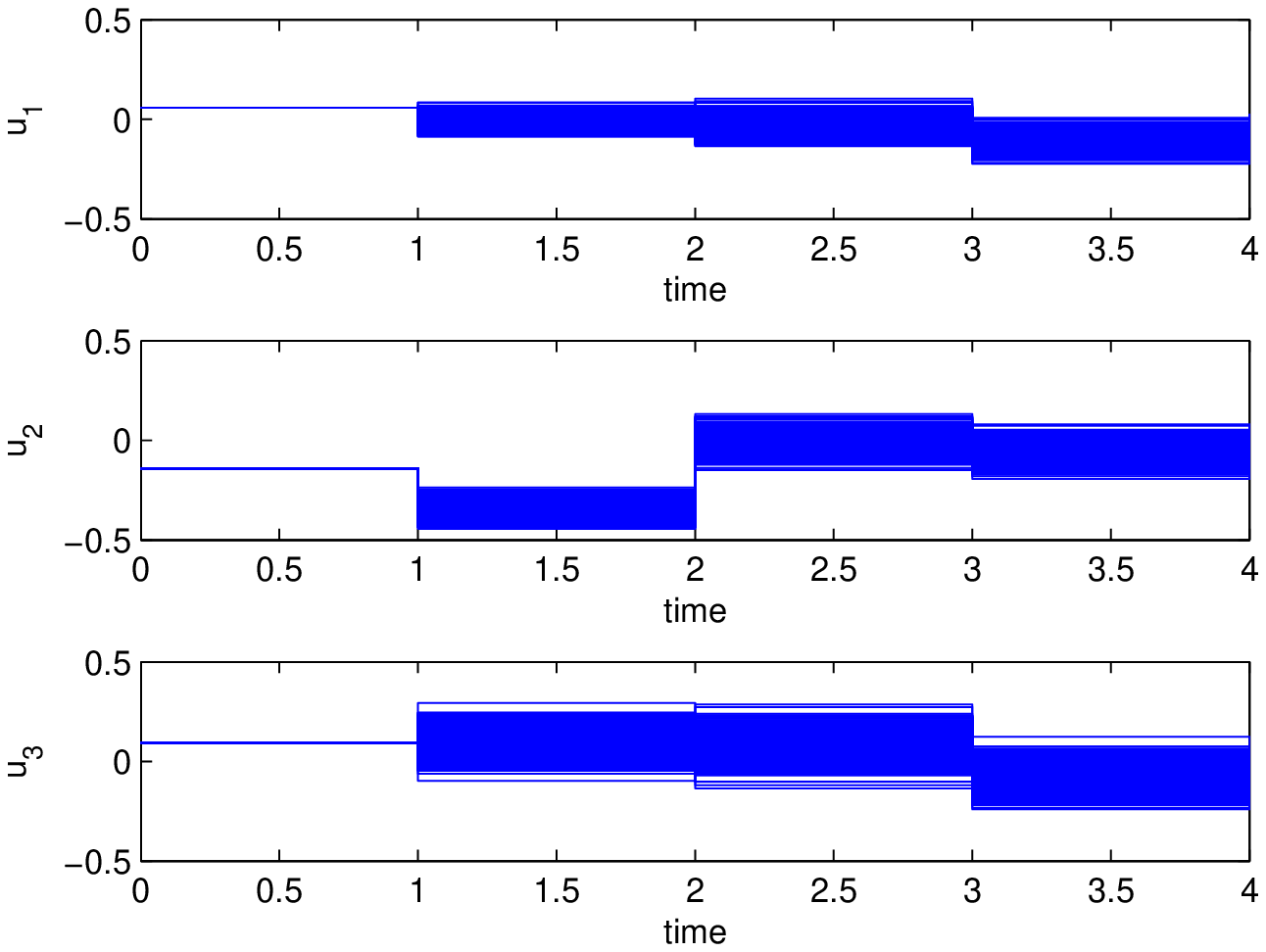}
\includegraphics[keepaspectratio=true,width=11cm]{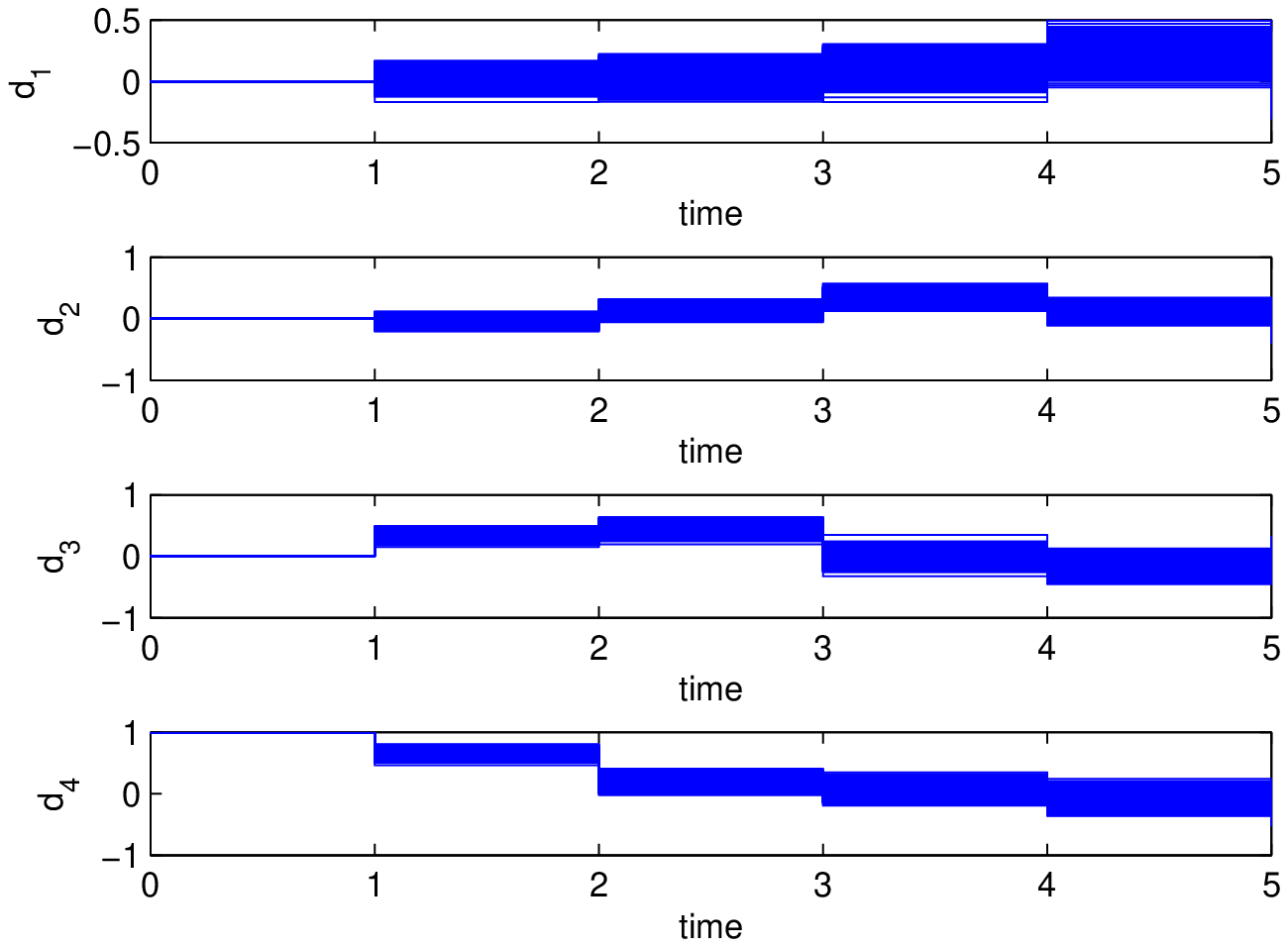}
\caption{1000 sample paths of the system with LQG control policy. Above: control input. Below: mass displacements.}\label{f:LQG}
\end{figure}
We shall now introduce constraints on the state and the control
input and study the feasibility of the problem with the methods of
Section~\ref{sec:cc}.
The convex approximations to the chance-constrained optimization
problems are solved numerically in Matlab by the toolbox
\texttt{CVX}~\cite{CVX}. In all cases we shall compute a $5$-stage
affine optimal control policy and apply it to repeated runs of the
system. Based on this we will discuss the feasibility of the hard
constrained problem and the probability of constraint violation.

\subsection{Polytopic constraints}

Let us impose bounds on the control inputs, $|u_1(t)|\leq 0.1$,
$|u_2(t)|\leq 0.3$ and $|u_3(t)| \leq 0.15$, with $t = 0, \ldots,
N-1$, and bounds on the mass displacements, $|d_i(t)| \leq 10$,
for $i=1,\cdots,4$ and with $t=1,\ldots, N$. In the notation of
Section~\ref{s:polyt}, these constraints are captured by the
equation $\eta(\bar x, \bar u) = T^x\bar x + T^u\bar u - y \leq 0$
where $T^x = \begin{bmatrix} M^T & 0\end{bmatrix}^T$ and $T^u =
\begin{bmatrix} 0 & H^T\end{bmatrix}^T$, with
$$M=\begin{bmatrix}M_1 & & \\ & \ddots &\\ & & M_1\end{bmatrix},\quad M_1 =  \begin{bmatrix} I_{4 \times 4} & 0_{4 \times 4} \\
-I_{4 \times 4} & 0_{4 \times 4} \end{bmatrix},\quad H=\begin{bmatrix}
I \\-I\end{bmatrix},$$ and
$$y=\begin{bmatrix} y_1 \\ y_2\end{bmatrix},\quad y_1 = \begin{bmatrix}
10 \\ \vdots \\ 10 \end{bmatrix},\quad y_2 =\begin{bmatrix}y'\\ \vdots
\\y' \end{bmatrix}, \quad y' = \begin{bmatrix}0.10 \\0.30 \\0.15
\end{bmatrix}.$$
This hard constraint is relaxed to the probabilistic constraint
$\PP[\eta(\bar x, \bar u)\leq 0]\geq 1-\alpha$. The resulting
optimal control problem is then addressed by constraint separation
~(Section~\ref{s:sep}) and ellipsoidal
approximation~(Section~\ref{s:confellipsoids}). \par
With
constraint separation, the problem is feasible for $\alpha\geq
0.05$. For $\alpha = 0.1$, the application of the suboptimal
control policy computed as in Proposition~\ref{prop:constrsep}
yields the results shown in Figure~\ref{f:constrsep}.
\begin{figure}
\centering
\includegraphics[keepaspectratio=true,width=11cm]{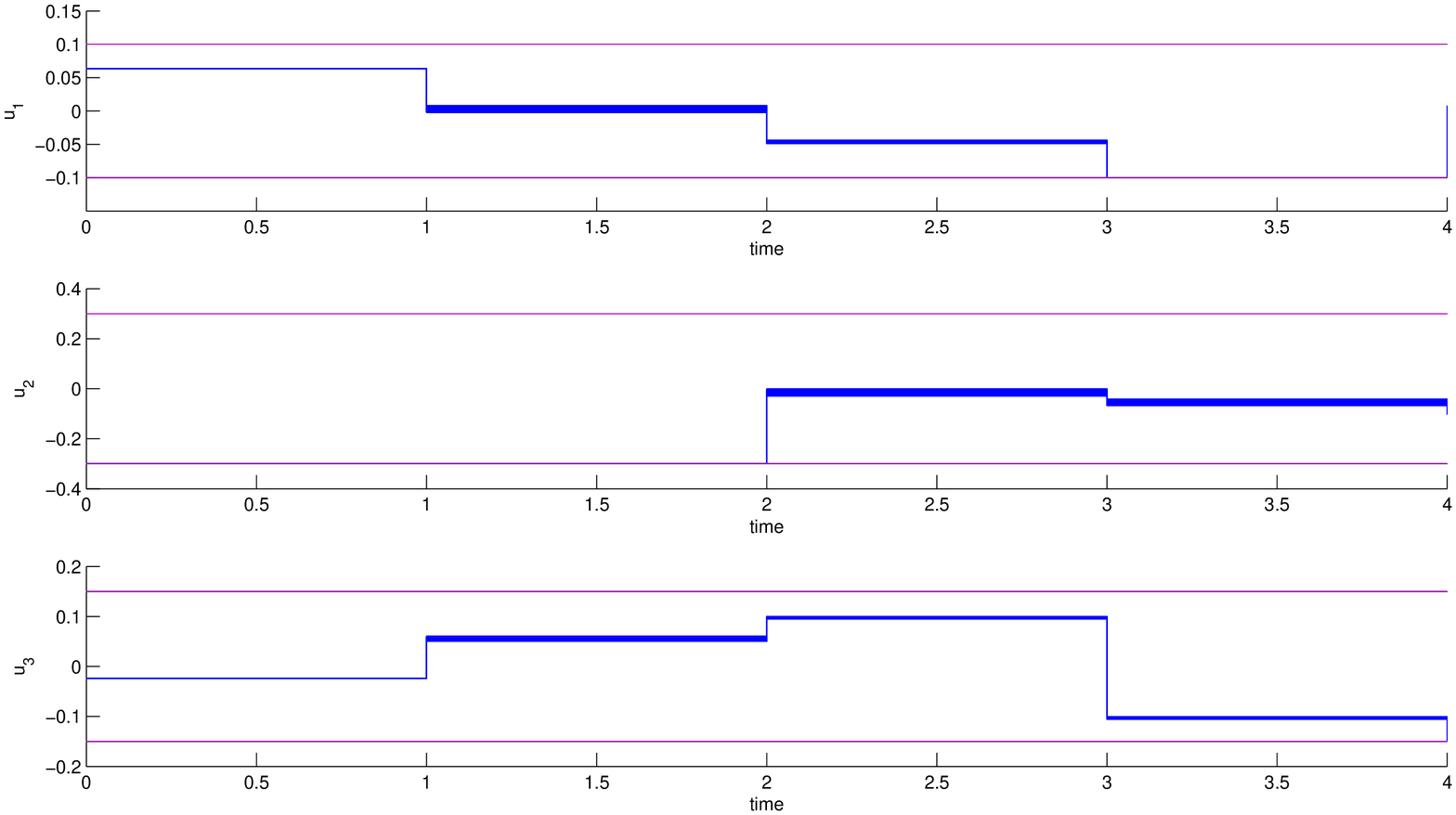}
\includegraphics[keepaspectratio=true,width=11cm]{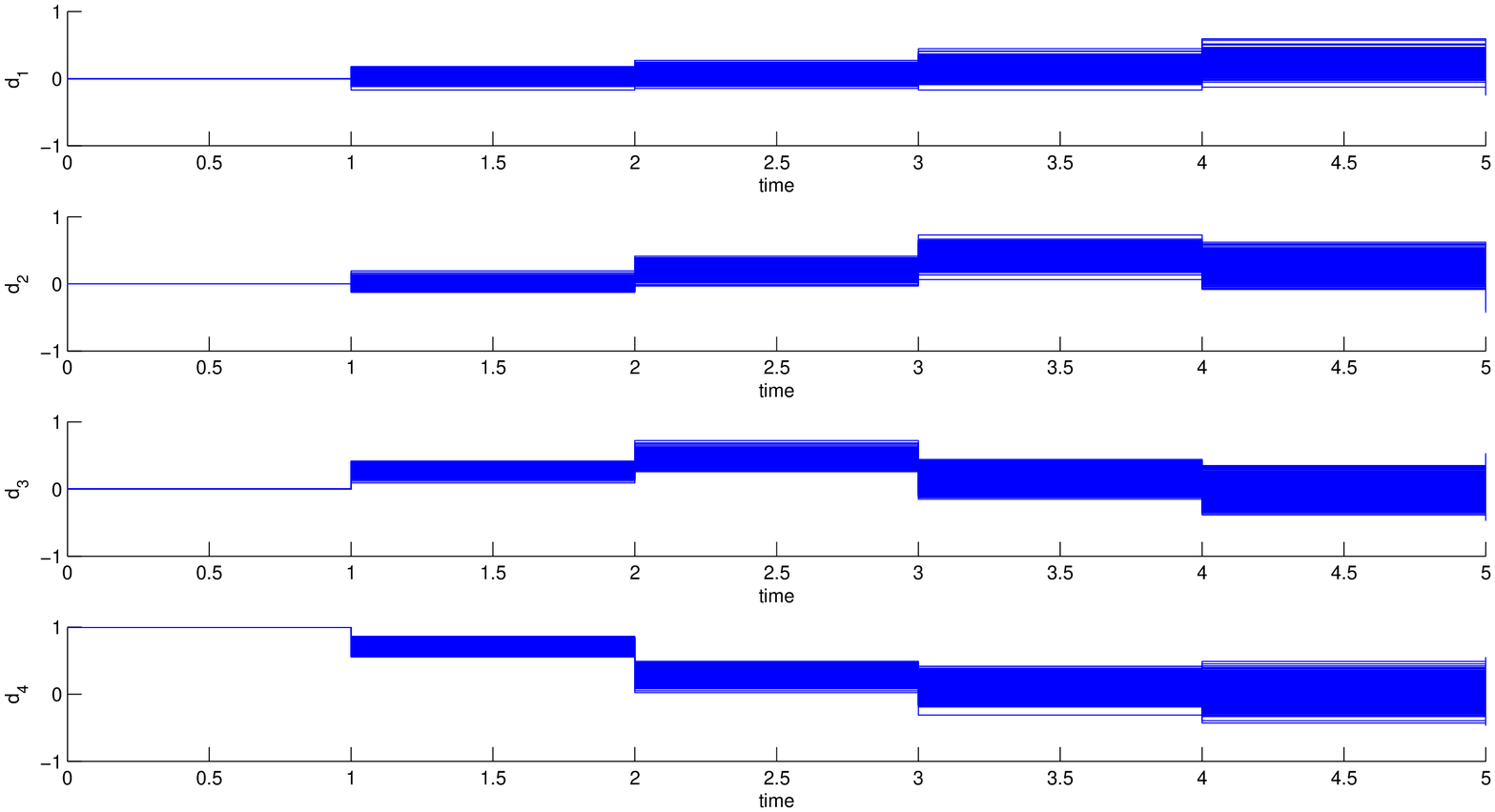}
\caption{1000 sample paths of the system with control policy computed
via constraint separation. Above: control input. Below: mass displacements. Horizontal straight lines show bounds.}\label{f:constrsep}
\end{figure}
With this policy, the control input saturates within the required
bounds whereas the mass displacements stay well below bounds. In
fact, although the required probability of constraint satisfaction
is $0.9$, constraints were never violated in 1000 simulation runs.
This suggests that the approximation incurred by constraint
separation is quite conservative, mainly due to the relatively
large number of constraints. It may also be noticed that the
variability of the applied control input is rather small. This
hints that the computed control policy is essentially open-loop,
i.e. the linear feedback gain is small compared to the affine
control term.
\par
With the ellipsoidal approximation method, for the same
probability level, the problem turns out to be infeasible, in
accordance with the conclusions of Section~\ref{s:comparison}. For
the sake of investigation, we loosened the bounds on the mass
displacements to $|d_i(t)| \leq 100$ for all $i$ and $t$. The
problem of Proposition~\ref{prop:confellip} is then feasible and
the results from simulation of the controlled system are reported
in Figure~\ref{f:confellip}.
\begin{figure}
\centering
\includegraphics[keepaspectratio=true,width=11cm]{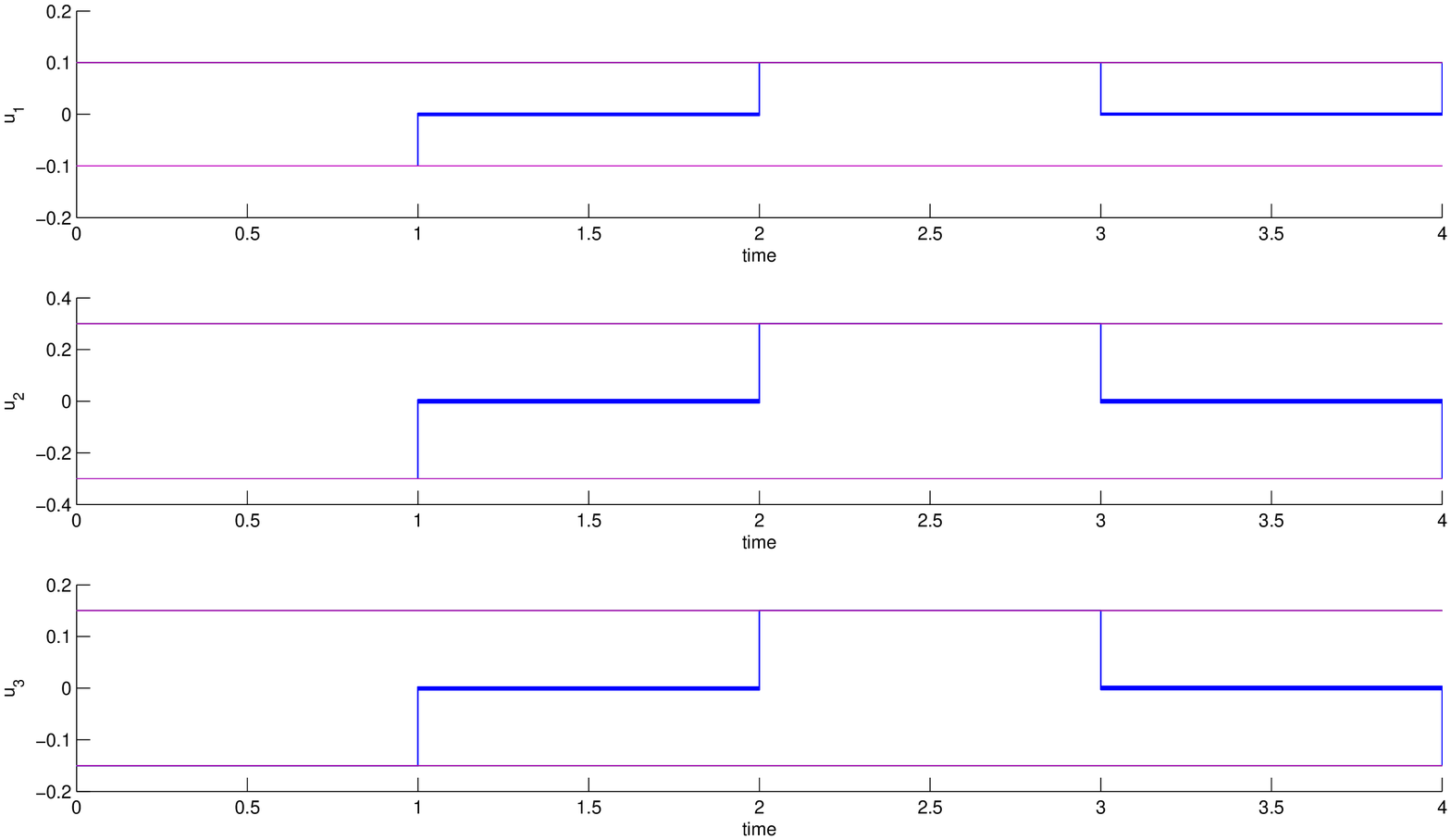}
\includegraphics[keepaspectratio=true,width=11cm]{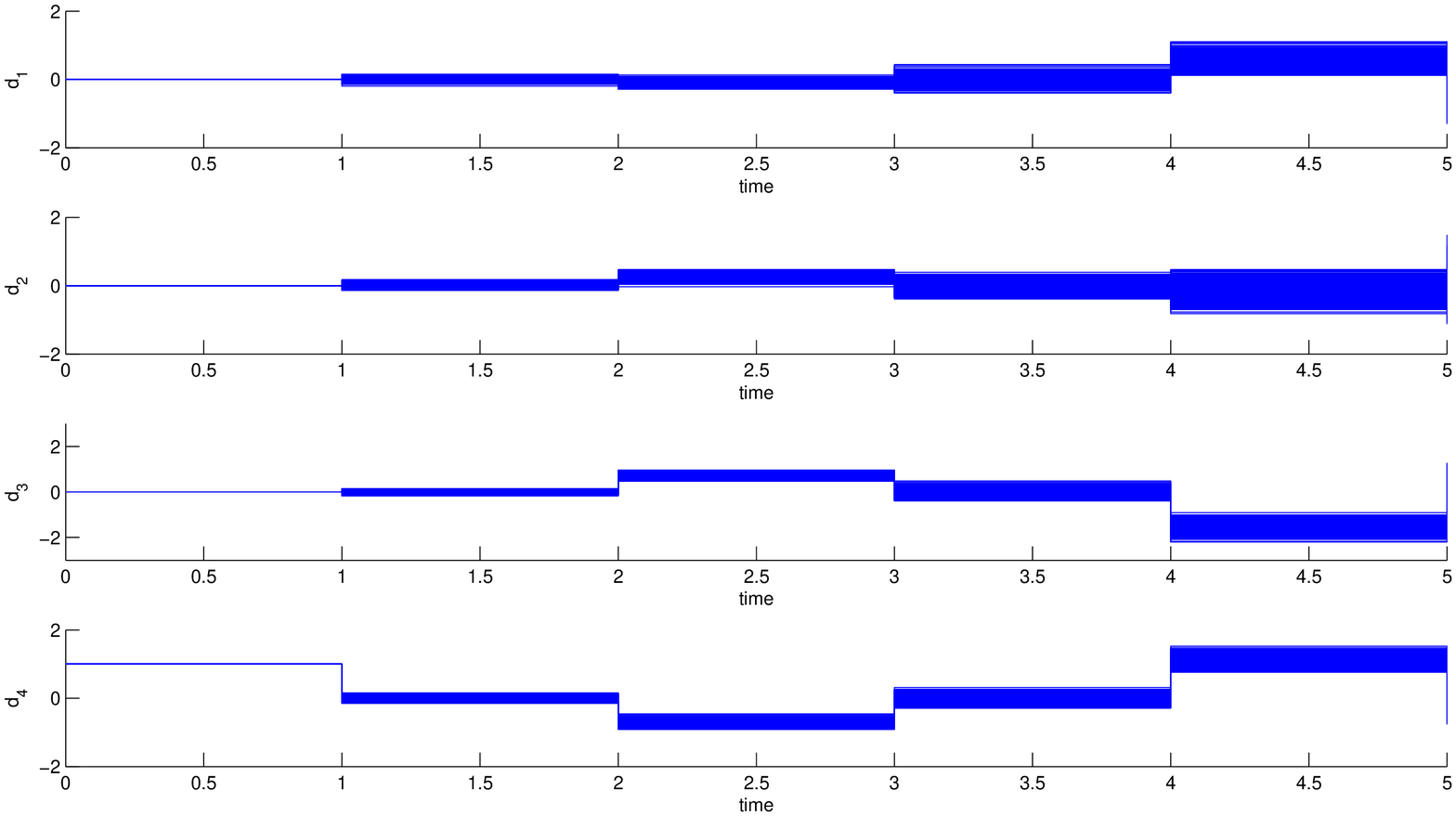}
\caption{1000 sample paths of the system with control policy computed
via ellipsoidal approximation. Above: control input. Below: mass displacements. Horizontal straight lines show bounds.}\label{f:confellip}
\end{figure}
Although the controller has been computed under much looser
bounds, the control performance is similar to the one obtained
with constraint separation, a clear sign that the ellipsoidal
approximation is overly conservative in this case. Another
evidence of inaccuracy is the fact that, while the control inputs
get closer to the bounds, the magnitude of the displacements is
not reduced. As in the case of constraint separation, the applied
control input is insensitive to the specific simulation run, i.e.
the control policy is essentially open loop.


\subsection{Ellipsoidal constraints}


Consider the constraint function
$$\sum_{k=1}^{N}\Bigl(d_1(k)^2+d_2(k)^2+d_3(k)^2+d_4(k)^2\Bigr)+\norm{\bar
u}^2\leq N\cdot c$$ with $c = (1^2
+1^2+1^2+1^2+0.1^2+0.15^2+0.3^2) = 4.1225$. Unlike the previous
section, we do not impose bounds on $d_i(k)$ and $u_i(k)$ at each
$k$ but instead require that the total ``spending'' on $x$ and $u$
does not exceed a total ``budget''. This constraint can be
modelled in the form of Section~\ref{sec:ellipconstr}, namely
$$\eta(\bar x, \bar u) = \Biggl(\begin{bmatrix} \bar x \\ \bar u
\end{bmatrix} -\delta \Biggr)^T \Xi \Biggl(\begin{bmatrix} \bar x
\\ \bar u
\end{bmatrix} -\delta \Biggr) -1\leq 0,$$
with $\delta=0$ and $\Xi=\frac{1}{N\cdot c}\begin{bmatrix}L & 0\\0
& I
\end{bmatrix}$, where
$$L=\begin{bmatrix}L_1 & & \\ & \ddots &\\ & & L_1\end{bmatrix},
L_1 =  \begin{bmatrix} I_{4 \times 4} & 0_{4 \times 4} \\ 0_{4
\times 4} & 0_{4 \times 4} \end{bmatrix}.$$ The constrained
control policy for $N=5$ and $\alpha=0.1$ is computed by solving
the LMI problem of Proposition~\ref{prop:ellipLMI}. Results from
simulations of the closed-loop system are reported in
Figure~\ref{f:ellipconstr}. Once again, constraints were not
violated over 1000 simulated runs, showing the conservatism of the
approximation. It is interesting to note that the displacements of
the masses are generally smaller than those obtained by the
controller computed under affine constraints, at the cost of a
slightly more expensive control action. In contrast with the
affine constraints case, the control action obtained here is much
more sensitive to the noise in the dynamics, i.e. the feedback
action is more pronounced.

\begin{figure}
\centering
\includegraphics[keepaspectratio=true,width=11cm]{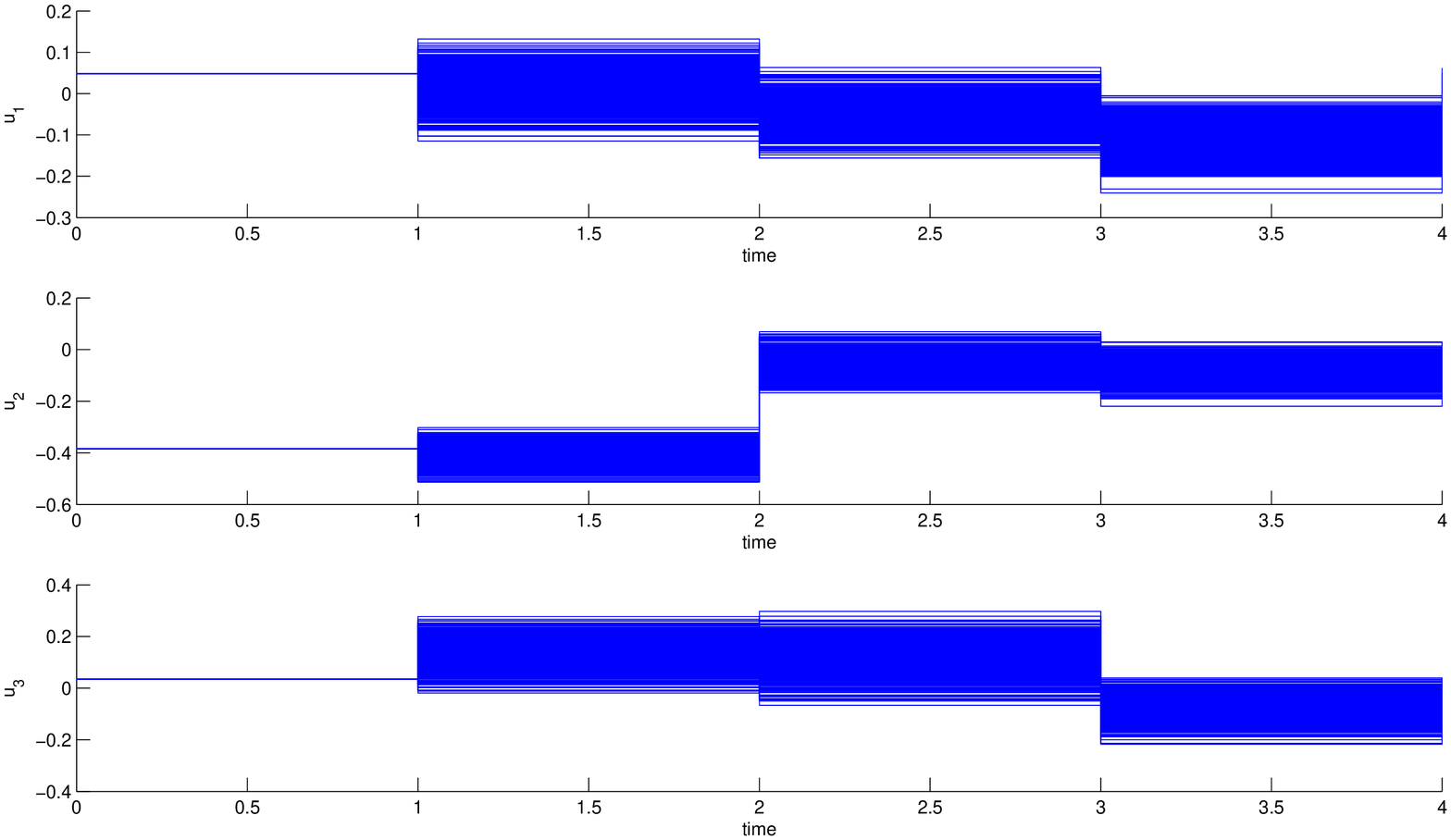}
\includegraphics[keepaspectratio=true,width=11cm]{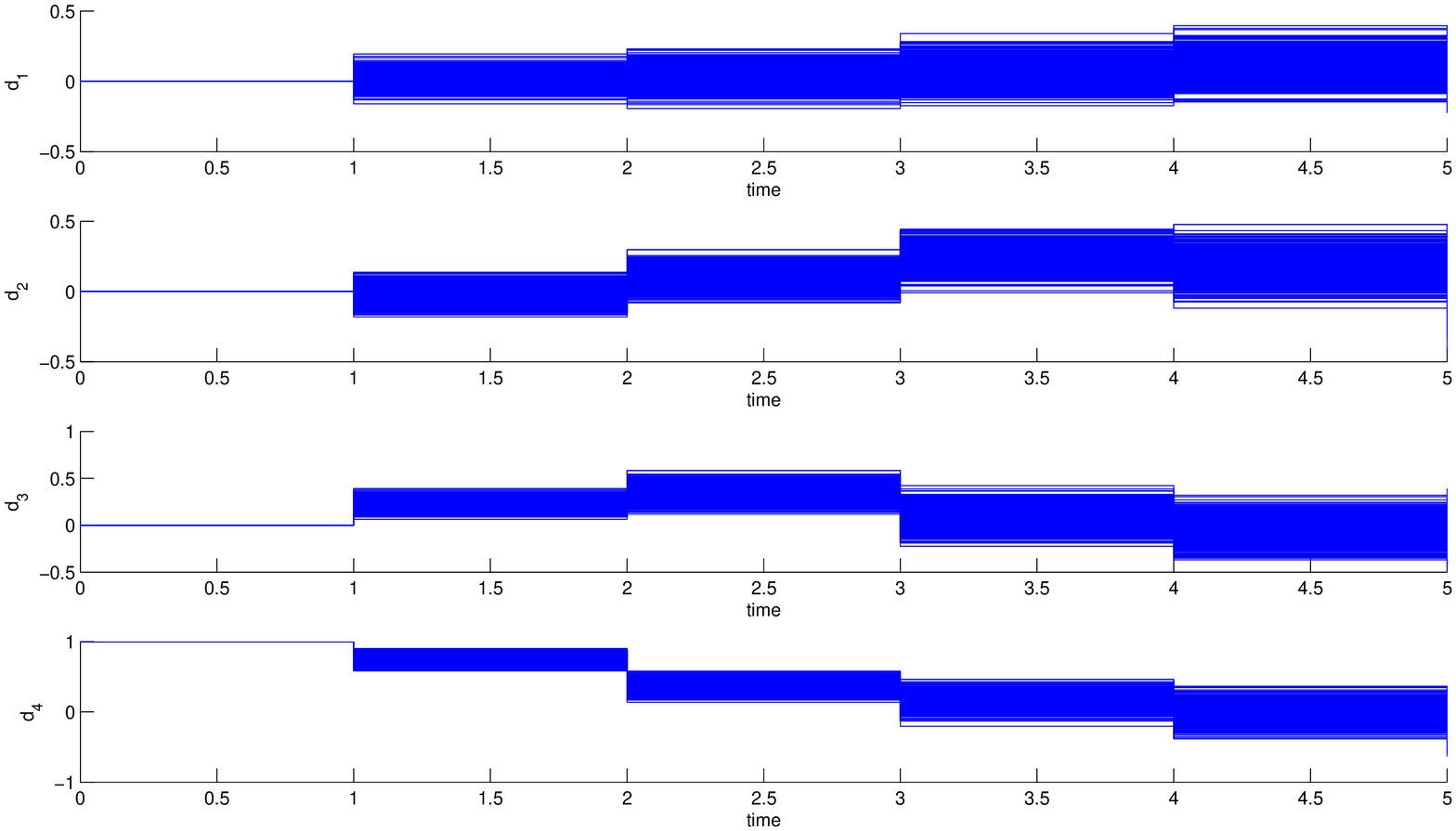}
\caption{1000 sample paths of the system with control policy computed
via ellipsoidal constraints. Above: control input. Below: mass displacements. Horizontal straight lines show bounds.}\label{f:ellipconstr}
\end{figure}

\section{Conclusions}
\label{section:conclusions}

We have studied the convexity of optimization problems with
probabilistic constraints arising in model predictive control of
stochastic dynamical systems. We have given conditions for the
convexity of expectation-type objective functions and constraints.
Convex approximations have been derived for nonconvex
probabilistic constraints. Results have been exemplified by a
numerical simulation study.\par Open issues that will be addressed
in the future are the role of the tunable parameters (e.g. the
$\alpha_i$ in Section~\ref{s:sep}, the $t_i$ of
Section~\ref{section:approxviaexp} and the $\beta_i$ in
Section~\ref{sec:icc}) in the various optimization problems, and
the effect of different choices of the ICC functions $\varphi_i$
(Section~\ref{sec:icc}). Directions of future research also
include the extension of the results presented here to the case of
noisy state measurements, the exact or approximate solution of the
stochastic optimization problems in terms of explicit control laws
and the control of stochastic systems with probabilistic
constraints on the state via bounded control laws.


\begin{thebibliography}{10}

\bibitem{ECC}
{\sc M.~Agarwal, E.~Cinquemani, D.~Chatterjee, and J.~Lygeros}, {\em On
  convexity of stochastic optimization problems with constraints}, in
  Proceedings of the European Control Conference (ECC'09), 2009.
\newblock Accepted.

\bibitem{aldenRH}
{\sc J.~M. Alden and R.~L. Smith}, {\em Rolling horizon procedures in
  nonhomogeneous {M}arkov decision processes}, Operations Research, 40 (1992),
  pp.~S183--S194.

\bibitem{astrom}
{\sc K.~{\AA}str{\"o}m}, {\em Introduction to Stochastic Control Theory},
  Academic Press, New York, 1970.

\bibitem{batinaPhDthesis}
{\sc I.~Batina}, {\em Model predictive control for stochastic systems by
  randomized algorithms}, PhD thesis, Technische Universiteit Eindhoven, 2004.

\bibitem{Ben-TaletAl}
{\sc A.~Ben-Tal, S.~Boyd, and A.~Nemirovski}, {\em Extending scope of robust
  optimization}, Mathematical Programming, 107 (2006), pp.~63--89.

\bibitem{ref:ben-tal}
{\sc A.~Ben-Tal, A.~Goryashko, E.~Guslitzer, and A.~Nemirovski}, {\em
  Adjustable robust solutions of uncertain linear programs}, Mathematical
  Programming, 99 (2004), pp.~351--376.

\bibitem{bertsekas}
{\sc D.~P. Bertsekas}, {\em Dynamic programming and suboptimal control: A
  survey from {ADP} to {MPC}}, European Journal of Control, 11 (2005).

\bibitem{bertsimas2007}
{\sc D.~Bertsimas and D.~B. Brown}, {\em Constrained stochastic {LQC}: a
  tractable approach}, IEEE Transactions on Automatic Control, 52 (2007),
  pp.~1826--1841.

\bibitem{bertsim}
{\sc D.~Bertsimas and M.~Sim}, {\em Tractable approximations to robust conic
  optimization problems}, Mathematical Programming, 107 (2006), pp.~5--36.

\bibitem{ref:bogachev}
{\sc V.~I. Bogachev}, {\em Gaussian Measures}, vol.~62 of Mathematical Surveys
  and Monographs, American Mathematical Society, Providence, RI, 1998.

\bibitem{ref:boyd04}
{\sc S.~Boyd and L.~Vandenberghe}, {\em Convex Optimization}, Cambridge
  University Press, Cambridge, 2004.

\bibitem{HitTime}
{\sc D.~Chatterjee, E.~Cinquemani, G.~Chaloulos, and J.~Lygeros}, {\em
  Stochastic control up to a hitting time: optimality and rolling-horizon
  implementation}.
\newblock \url{http://arxiv.org/abs/0806.3008}, jun 2008.

\bibitem{ReachAvoid}
{\sc D.~Chatterjee, E.~Cinquemani, and J.~Lygeros}, {\em Maximizing the
  probability of attaining a target prior to extinction}.
\newblock \url{http://arxiv.org/abs/0904.4143}, apr 2009.

\bibitem{BoundedControl}
{\sc D.~Chatterjee, P.~Hokayem, and J.~Lygeros}, {\em Stochastic model
  predictive control with bounded control inputs: a vector space approach}.
\newblock \url{http://arxiv.org/abs/0903.5444}, mar 2009.

\bibitem{kouvaritakissMPCIneqconstraints}
{\sc P.~D. Couchman, M.~Cannon, and B.~Kouvaritakis}, {\em Stochastic {MPC}
  with inequality stability constraints}, Automatica J. IFAC, 42 (2006),
  pp.~2169--2174.

\bibitem{dudley}
{\sc R.~M. Dudley}, {\em Real Analysis and Probability}, vol.~74 of Cambridge
  Studies in Advanced Mathematics, Cambridge University Press, Cambridge, 2002.
\newblock Revised reprint of the 1989 original.

\bibitem{goulart}
{\sc P.~J. Goulart, E.~C. Kerrigan, and J.~M. Maciejowski}, {\em Optimization
  over state feedback policies for robust control with constraints},
  Automatica, 42 (2006), pp.~523--533.

\bibitem{CVX}
{\sc M.~Grant and S.~Boyd}, {\em {CVX}: Matlab software for disciplined convex
  programming}, feb 2009.
\newblock (web page and software) http://stanford.edu/$\sim$boyd/cvx.

\bibitem{haneveldICC}
{\sc W.~K.~K. Haneveld}, {\em On integrated chance constraints}, in Stochastic
  programming (Gargnano), vol.~76 of Lecture Notes in Control and Inform. Sci.,
  Springer, Berlin, 1983, pp.~194--209.

\bibitem{haneveldICC06}
{\sc W.~K.~K. Haneveld and M.~H. van~der Vlerk}, {\em Integrated chance
  constraints: reduced forms and an algorithm}, Computational Management
  Science, 3 (2006), pp.~245--269.

\bibitem{Lofberg:03}
{\sc J.~L{\"{o}}fberg}, {\em Minimax Approaches to Robust Model Predictive
  Control}, PhD thesis, Link{\"o}ping Universitet, Apr 2003.

\bibitem{maciejowski}
{\sc J.~Maciejowski}, {\em Predictive Control with Constraints}, Prentice Hall,
  2002.

\bibitem{nemshap}
{\sc A.~Nemirovski and A.~Shapiro}, {\em Convex approximations of chance
  constrained programs}, SIAM Journal on Control and Optimization, 17 (2006),
  pp.~969--996.

\bibitem{prekopa}
{\sc A.~Pr{\'e}kopa}, {\em Stochastic Programming}, vol.~324 of Mathematics and
  its Applications, Kluwer Academic Publishers Group, Dordrecht, 1995.

\bibitem{Stanfordgang}
{\sc J.~Primbs}, {\em A soft constraint approach to stochastic receding horizon
  control}, in Proceedings of the 46th IEEE Conference on Decision and Control,
  2007, pp.~4797 -- 4802.

\bibitem{primbs}
{\sc J.~Primbs}, {\em Stochastic receding horizon control of constrained linear
  systems with state and control multiplicative noise}, in Proceedings of the
  26th American Control Conference, 2007.

\bibitem{tempoetal}
{\sc G.~C. R.~Tempo and F.~Dabbene}, {\em Randomized Algorithms for Analysis
  and Control of Uncertain Systems}, Springer-Verlag, 2005.

\bibitem{spallbook}
{\sc J.~Spall}, {\em Introduction to Stochastic Search and Optimization:
  Estimation, Simulation, and Control}, Wiley, Hoboken, NJ, 2003.

\bibitem{vanHessemFullSolution}
{\sc D.~H. van Hessem and O.~H. Bosgra}, {\em A full solution to the
  constrained stochastic closed-loop mpc problem via state and innovations
  feedback and its receding horizon implementation}, in Proceedings of the 42nd
  IEEE Conference on Decision and Control, vol.~1, 9-12 Dec. 2003,
  pp.~929--934.

\bibitem{sagarpaper}
{\sc M.~Vidyasagar}, {\em Randomized algorithms for robust controller synthesis
  using statistical learning theory}, Automatica, 37 (2001), pp.~1515--1528.

\end{thebibliography}

\end{document}